\documentclass[11pt]{amsart}
\usepackage{amssymb,amsmath,amsthm}
\usepackage{mathrsfs,dsfont,a4wide}
\usepackage{xcolor}
\usepackage{hyperref}
\theoremstyle{plain}
\newtheorem{theorem}{Theorem}[section]
\newtheorem{lemma}[theorem]{Lemma}
\newtheorem{proposition}[theorem]{Proposition}

\newtheorem{remark}[theorem]{Remark}
\newtheorem{definition}[theorem]{Definition}
\theoremstyle{definition}
\theoremstyle{remark}
\numberwithin{equation}{section}
\usepackage{geometry}
\usepackage{graphicx,pgf,tikz} 
\usetikzlibrary{arrows}
\usetikzlibrary{snakes}
\tikzstyle xyax=[thin]
\tikzstyle mlin=[thick]
\tikzstyle slin=[]

\newcommand{\as}{{\mathcal A}}

\newcommand{\dd}{{\mathrm{d}}}
\newcommand{\Ha}{{\mathcal{H}}}
\newcommand{\J}{{\mathcal J}}

\newcommand{\E}{{\mathcal E}}
\newcommand{\R}{{\mathbb R}}

\newcommand{\N}{{\mathbb N}}

\newcommand{\loc}{\text{loc}}
\newcommand{\Om}{\Omega}
\newcommand{\Omb}{\overline{\Omega}}
\newcommand{\om}{\omega}
\newcommand{\eps}{\varepsilon}

\newcommand{\weakst}{\stackrel{\ast}{\rightharpoonup}}

\definecolor{verde}{RGB}{20,150,100}

\newcommand{\weak}{\rightharpoonup}

\begin{document}
\title
{A free discontinuity approach to optimal profiles in Stokes flows}
\author[D. Bucur]
{Dorin Bucur}
\address[Dorin Bucur]{Laboratoire de Math\'ematiques CNRS UMR 5127
Universit\'e de Savoie Mont Blanc Campus Scientifique 73 376 Le Bourget-Du-Lac, France}
\email[D. Bucur]{  dorin.bucur@univ-savoie.fr}
\author[A. Chambolle]
{Antonin Chambolle}
\address[Antonin Chambolle]{Ceremade, CNRS and Universit\'{e} de Paris-Dauphine PSL, Place de Lattre de Tassigny, 75775 Paris Cedex 16, France}
\email[A. Chambolle]{chambolle@ceremade.dauphine.fr}
\author[A. Giacomini]
{Alessandro Giacomini}
\address[Alessandro Giacomini]{DICATAM, Sezione di Matematica, Universit\`a degli Studi di Brescia, Via Branze 43, 25123 Brescia, Italy}
\email[A. Giacomini]{alessandro.giacomini@unibs.it}
\author[M. Nahon]
{Micka\"{e}l Nahon}
\address[Micka\"{e}l Nahon]{ Max-Planck-Institut f\"ur Mathematik in den Naturwissenschaften\\
04103 Leipzig,
Germany
}
\email[M. Nahon]{mickael.nahon@mis.mpg.de}
\begin{abstract}
In this paper we study obstacles immerged in a Stokes flow with Navier boundary conditions. We prove the existence and regularity of an obstacle with minimal drag, among all shapes of prescribed volume and controlled surface area, taking into account that these shapes may naturally develop geometric features of codimension $1$. The existence is carried  out  in the framework of free discontinuity problems and leads to a relaxed solution in the space of special functions of bounded deformation ($SBD$). In dimension $2$, we prove that the solution is classical. 
\end{abstract}

\maketitle
\tableofcontents

\section{Introduction}
\label{sec_intr}
Consider an obstacle $E\subset \R^d$ ($d=2,3$ in real applications) contained in a (finite) channel $\Om$ in which a fluid with viscosity coefficient $\mu>0$ is flowing. Assume that the flow is stationary and incompressible, and that the associated velocity field $u$ is equal to a constant vector $V_\infty$ on the walls of the channel. The obstacle $E$ experiences  a force, whose component in direction of $V_\infty$ will be denoted by 
$Drag(E)$ and is usually called the {\it drag force}. If we further assume that the velocity of the fluid satisfies the Stokes equation in $\Om \setminus E$ and obeys to {\it Navier boundary conditions} on $\partial E$, the expression of the drag force turns out to be given (up to a multiplicative constant) by 
\begin{equation}
\label{eq_drag-intr}
Drag(E)=2\mu\int_{\Om\setminus E}|e(u)|^2\,dx+\beta\int_{\partial E}|u|^2\,d\Ha^{d-1},
\end{equation}
where $e(u):=\frac{1}{2}(Du+(Du)^*)$ denotes the symmetrized gradient of $u$ and $\beta>0$ is the friction coefficient (we refer to Subsection \ref{subsec_drag} for details).
\par
We are interested in minimizing the drag force among all obstacles $E$ with a prescribed volume and controlled surface area. Precisely we look for the existence of such an optimal obstacle and  for  its qualitative properties. The existence question is not very relevant as soon as one imposes strong geometric constraints on the admissible obstacles (e.g. convexity, uniform cone conditions, etc.) since this may hide some specific features which would naturally occur. Indeed, letting the geometry of the obstacle to be completely free, some qualitative  behavior  (blocked by rigid geometric constraints) can be observed. This is the case of our problem, where the optimal obstacle (that we prove to exist without imposing any geometric or topological constraint) may be composed, roughly speaking as a union of a body with the prescribed volume and pieces of surfaces of dimension $d-1$. Those surfaces do not have volume, but count for the total surface area $\Ha^{d-1}(\partial E)$ and of course have a strong influence on the flow. 
\par
Penalizing the surface area and the volume, the model problem we are interested in can be written as 
\begin{equation*}
\label{eq_pb-intr}
\min_{E} \left\{Drag(E)+c\Ha^{d-1}(\partial E)+f(|E|)\right\},
\end{equation*}
where $c>0$ and $f:(0,|\Om|)\to\R\cup\{+\infty\}$ is a lower semicontinuous function. Roughly speaking, the terms involving perimeter and volume can be thought as a price to pay in order to build the obstacle $E$, and we can give the two relevant choices of function $f$:
\[
f(m)=+\infty 1_{\{m\neq m_0\}}\text{ for some }m_0\in (0,|\Om|),\text{ or }f(m)=-\lambda m\text{ for some }\lambda>0.
\]
\vskip10pt
Many similar optimisation problems have been considered under the ``no-slip'' boundary condition, meaning flows for which $u=0$ at $\partial E$. Under volume constraint and an a priori symmetry hypothesis  around an axis parallel to the flow, the minimal drag question has been studied in \cite{W71} on smooth surfaces. In \cite{P73}, still under symmetry hypotheses, it was conjectured that the optimal profile in three dimensions is a prolate spheroid with sharp ends of angle of $120$ degrees. In the same symmetry context, let us also mention the slender body approximation of \cite{T68}. We also refer the reader to the paper by \u{S}ver\'ak \cite{Sv93} who, in two dimensions, proves the existence of an optimal obstacle under  topological hypotheses, namely that the obstacle has at most a given number of connected components (in particular this number can be equal to $1)$. The  proof is genuinely two dimensional and can not be extended to higher dimensions. 
\par
The Navier boundary condition gives many new challenges, namely the possible apparition of lower dimensional structures in the obstacle that minimize the drag, something which was absent under the no-slip condition. The Navier boundary condition may be seen as a partial adherence to the boundary of the obstacle, and it may be asymptotically obtained as a limit of flows with {\it perfect slip} on an obstacle with rough boundary. More precisely, a periodic microstructure with the right scaling on the boundary is modelled at the limit by a Navier boundary condition, as was proved in \cite{CLS_10}. In dimension higher than two it is also necessary to take into account more complex geometries for the microstructure, which at the limit produce an anisotropic factor that favors certain directions of the flow. Moreover, infinitesimal boundary perturbations can dramatically modify the solution of the Stokes equation with Navier boundary conditions, while in presence of no-slip boundary conditions the solution remains stable.  We refer the reader to \cite{BFN10} for an analysis of those phenomena and for a discussion on the pertinence of the Navier boundary conditions in physical models.

\vskip10pt
For a fixed obstacle $E$, the minimization of the drag with respect to the friction parameter $\beta$ of the Navier conditions (meaning, from a physical point of view, with respect to the microstructure on the boundary) has been studied in \cite{BB_13}, for both Stokes and Navier-Stokes flows. While for Stokes flows the drag is increasing with the friction parameter, an important observation which occurs for the Navier-Stokes equation is that the monotonicity of the drag with respect to the  parameter $\beta$ does not hold. This is a reason for which the results we give in this paper for the Stokes flows are not expected to hold, as such, for the Navier-Stokes equation.

\vskip10pt
 Since the stationary velocity field associated to a Lipschitz obstacle $E$ turns out to be characterized variationally as the minimizer of the right hand side of \eqref{eq_drag-intr} in the class of admissible velocities 
\[
\mathcal{V}^{\text{reg}}_{E,V_\infty}(\Om)=\left\{u\in H^1(\Om\setminus E):\text{div} u=0,\ u_{|\partial E}\cdot\nu_E=0,\ u_{|\partial\Om}=V_\infty\right\}
\]
(see \eqref{eq_def-vub} in Subsection \ref{subsec_flow} for more details), we can conveniently rephrase the minimization problem by letting also the velocity fields intervene explicitely in the form
\par
\begin{equation}
\label{eq_pb-min-intr}
\min_{E,u\in \mathcal{V}^{\text{reg}}_{E,V_\infty}(\Om)} \left\{2\mu\int_{\Om\setminus E}|e(u)|^2\,dx+\beta\int_{\partial E}|u|^2\,d\Ha^{d-1}+c\Ha^{d-1}(\partial E)+f(|E|)\right\}.
\end{equation}
The first main goal of the paper is to find suitable relaxations of problem \eqref{eq_pb-min-intr} for which we can prove the existence of minimizers  without any a priori constraint on the regularity or the topology of the sets $E$.

\vskip10pt
In order to avoid unnatural geometric restrictions on the obstacle $E$, it is natural in view of the third term appearing in \eqref{eq_pb-min-intr} to let it vary within the class of sets of {\it finite perimeter} (see Subsection \ref{subsec_bv}), and replace the topological boundary with reduced one $\partial^*E$. 
\par
In order to describe properly obstacles with very narrow spikes which in the limit degenerate to $(d-1)$-surfaces and that cannot be taken into account through the reduced boundary, it is convenient to consider admissible velocity fields which can be {\it discontinuous} outside $E$ (see Subsection \ref{subsec_opt-pb}). Since the symmetrized gradient $e(u)$ is involved explicitly in \eqref{eq_pb-min-intr}, a natural family for the admissible velocities is given by the space of {\it functions of bounded deformation} $SBD$. The natural relaxation of the energy takes the form  
\begin{equation}\label{eq_JEu-intr}
\begin{split}
\J(E,u):=&2\mu\int_{\Om\setminus E} |e(u)|^2\,dx+\beta\int_{\partial^*E}|u^+|^2\,d\Ha^{d-1}+\beta\int_{J_u\setminus \partial^*E}[|u^+|^2+|u^-|^2]\,d\Ha^{d-1}\\
&+c\Ha^{d-1}(\partial^*E)+2c\Ha^{d-1}(J_u\setminus \partial^*E)+f(|E|),
\end{split}
\end{equation}
 where $u$ is set equal to zero a.e. in $E$, while $J_u$ denotes the discontinuity set of $u$ and $u^\pm$ are the traces of $u$ on $\partial^*E$ and $J_u$ (the trace $u^-$ vanishes on $\partial^*E$ by the choice of orientation, while $u^+$ is on the outward side).
\par
Within this framework the {\it global} obstacle is given by $E \cup J_u$, so that it contains also {\it lower dimensional parts}, namely $J_u\setminus\partial^*E$:  roughly speaking, for the optimal velocity these discontinuous regions generate $(d-1)$-surfaces which correspond to volumeless,  lower dimensional subsets of the optimal obstacle. 
\par
Admissible velocities must be tangent to the obstacles, meaning that  not only $u$ is tangent to $\partial^*E$, but also the two traces  $u^\pm$ are orthogonal to the normal $\nu_u$ along the jump set $J_u$. The contribution of the {\it Navier surface term} takes naturally into account the contribution from both sides given by $u^\pm$.    Concerning the perimeter term, we count {\it twice} the lower dimensional parts because we see the relaxed obstacle as a limit of regular obstacles, such that points of $J_u\setminus \partial^*E$ correspond to thin parts of the regular obstacle that collapse on a lower-dimensional structure. We could also see the perimeter term as a price to pay in order to construct the obstacle and just keep $\Ha^{d-1}(\partial^*E\cup J_u)$ instead, and the main results of the paper would not be affected.
\par
The  relaxed optimization problem can be seen as a minimization problem on the pairs $(E,u)$ which has the features of classical geometrical problems for $E$ coupled with a {\it free discontinuity problem} for $u$, with a surface term depending on the traces which are subject to suitable tangency constraints and boundary conditions.

\vskip15pt

The first main results of the paper (Theorem \ref{main1}) concerns the existence of minimizers for the relaxed functional $\J$ in \eqref{eq_JEu-intr} among the class of admissible configurations (see Definition \ref{def:adm-Eu} for the precise definition).

The main difficulties we have to face in order to prove that the problem is well posed are the following:
\begin{itemize}
\item[(a)] the closure of the {\it non-penetration} constraint for the velocity on $\partial^*E\cup J_u$ under the natural weak convergence of the problem;
\item[(b)] the lower semicontinuity of energies of the form
\begin{equation}
\label{eq_nav-intr}
\int_{J_u}[|u^+|^2+|u^-|^2]\,d\Ha^{d-1}
\end{equation}
associated to the Navier conditions.
\end{itemize}
Point (a) is a consequence of a lower semicontinuity result for the energy 
$$
\int_{J_u}\left[ |u^+\cdot \nu_u|+|u^-\cdot \nu_u|\right]\,d\Ha^{d-1}
$$
which is proved in Theorem \ref{thm:sci-normal}, by resorting to recent lower semicontinuity results for functionals on $SBD$ by Friedrich, Perugini and Solombrino \cite{FPS}.
\par
The energy of point (b) naturally appears in a scalar setting when dealing with shape optimization problems involving {\it Robin boundary conditions} (see e.g. \cite{BucGiac, bugi15, bugi2015, CaKri}), and it is easily seen to enjoy lower semicontinuity properties by working with sections. The lower semicontinuity result in the vectorial SBD setting is given by Theorem \ref{thm:lsc-phi} and cannot rely on an easy argument by sections, which instead would yield the lower semicontinuity of an energy of the form
$$
\int_{J_u}\left[ |u^+\cdot \xi|^2+|u^-\cdot \xi|^2\right]|\xi\cdot \nu_u|\,d\Ha^{d-1}
$$
with $\xi\in \R^d$ with $|\xi|=1$: the optimization in $\xi$ in order to recover \eqref{eq_nav-intr} does not seem feasible in dimension $d\geq 3$. We thus follow a different strategy based on a blow up argument in which we reconstruct the vector quantities $u^\pm$ by controlling them along a sufficiently high number of directions (see Subsection \ref{subsec_lsc-upm} for details): in this way we can deal with more general energy densities of the form $\phi(u^+)+\phi(u^-)$, where $\phi$ is a lower semicontinuous function.

\vskip10pt
The second main result of the paper (see Theorem \ref{main2}) concerns the regularity of the relaxed minimizers of \eqref{eq_JEu-intr}.  Provided that the  volume penalization function $f$ is Lipschitz and that we are in two dimensions, we prove that  for a minimizer $(E,u)$ of $\J$, the optimal obstacle $E\cup J_u$ is a closed set, while the optimal velocity $u$ is a smooth Sobolev function outside the obstacle, recovering somehow the classical setting of the problem. More precisely we show that
\begin{equation}
\label{eq:closed-intr}
\Ha^1(\Om\cap\overline{\partial^*E\cup J_u}\setminus (\partial^*E\cup J_u))=0,
\end{equation}
so that the optimal obstacle can be described as the closed set obtained by the complement of the connected components of $\Om\setminus\overline{\partial^*E\cup J_u}$ on which $u$ does not vanish identically.

 The technical ideas to prove \eqref{eq:closed-intr} stem  from the pioneering result of De Giorgi, Carriero and Leaci on the Mumford-Shah problem \cite{DGCL}, where the key of the proof is a decay estimate obtained by a contradiction/compactness argument. For vectorial problems, a similar strategy, but definitely more involved, was used for the Griffith fracture problem in \cite{CFI_19} (for the two-dimensional case) and in \cite{CCI_19} (for higher dimension). In the fracture problem, the key compactness result relies on the possibility to approximate a field $u\in SBD([-1,1]^d)$ with a small jump set by a Sobolev function which is locally controlled in $H^1$ (via the classical Korn inequality).
\par
In our case,  we follow a similar approximation procedure,  but we have to handle two additional constraints: incompressibility and non-penetration at the jumps. From a technical point of view, this  is problematic since the bound in \cite{CFI_19} in not strong enough to stay in divergence-free vector fields and the method in \cite{CCI_19} creates new jumps on which the non-penetration constraint is not a priori verified. However, when restricted to two dimensions, the method of \cite{CCI_19} leads to a stronger result,  so that  both constraints can be handled. \newline

The paper is organized as follows. In Section \ref{sec_prel} we recall fix the notation and recall some basic facts concerning sets of finite perimeter, functions of bounded deformation and Hausdorff convergence of compact sets. Section \ref{sec_stokes} is devoted to the precise exposition of the drag optimization problem. In Section \ref{sec_ex-ob} we detail the relaxation of the problem in the family of obstacle of finite perimeter and with velocities of bounded deformation, and formulate  the main results of the paper concerning the existence of minimizers (in any dimension) and their regularity in dimension two. The proof of the existence of minimizers is given in Section \ref{sec_proof-main}, and it is based on some technical results for $SBD$ functions collected in Section \ref{sec_sbd-tech}, while the regularity result is proved in Section \ref{sec-reg}.

\section{Notations and Preliminaries}
\label{sec_prel}

\subsection{Basic notation.} If $E \subseteq \R^d$, we will denote with $|E|$ its $d$-dimensional Lebesgue measure, and by $\Ha^{d-1}(E)$ its $(d-1)$-dimensional Hausdorff measure: we refer to \cite[Chapter 2]{EvansGariepy} for a precise definition, recalling that for sufficiently regular sets $\Ha^{d-1}$ coincides with the usual area measure. Moreover, we denote by $E^c$ the complementary set of $E$, and by $1_E$ its characteristic function, i.e., $1_E(x)=1$ if $x \in E$, $1_E(x)=0$ otherwise.  In addition we will say that $E_1\Subset E_2$ if $\overline{E_1}\subset E_2$. Finally we will denote with $Q_{x,r}\subseteq \R^d$ the cube of center $x$ and side $r$: when $x=0$, we will simply write $Q_r$.  
\par
If $A \subseteq \R^d$ is open and $1 \leq p \leq +\infty$, we denote by $L^p(A)$ the usual space of $p$-summable functions on $A$ with norm indicated by $\|\cdot\|_p$. $W^{1,p}(A)$ will stand for the Sobolev space of functions in $L^p(A)$ whose gradient in the sense of distributions belongs to $L^p(A;\R^d)$. Finally, given a finite dimensional unitary space $Y$, we will denote by $\mathcal{M}_b(A;Y)$ will denote the space of $Y$-valued Radon measures on $A$, which can be identified with the dual of $Y$-valued continuous functions on $A$ vanishing at the boundary.
\par
We will denote by $M^{d\times m}$ the set of $d\times m$ matrices with values in $\R$: when $d=m$ we will denote by ${\rm M}^{d\times d}_{sym}$ the subspace of $d\times d$ symmetric matrices. For $a\in \R^d$ and $b\in \R^m$ we will denote with $a\otimes b$ the element of $M^{d\times m}$ such that
$$
(a\otimes b)_{ij}=a_ib_j,
$$
while if $a,b\in\R^d$ we denote with $a\odot b$ the matrix in ${\rm M}^{d\times d}_{sym}$ such that
$$
(a\odot b)_{ij}=\frac{a_ib_j+a_jb_i}{2}.
$$
\par
Given $\xi\in \R^d$ with $|\xi|=1$, we denote with $\xi^\perp$  the hyperplane through the origin orthogonal to $\xi$. If $E\subseteq \R^d$, we set
\begin{equation}\label{eq:Exi}
E^\xi:=\pi_{\xi^\perp}(E),
\end{equation}
where $\pi$ denotes the orthogonal projection, and for $y\in \xi^\perp$ we set
\begin{equation}
\label{eq_Exiy}
E^\xi_y:=\{t\in\R\,:\, y+t\xi\in E\}.
\end{equation}
 
\vskip10pt\noindent
\subsection{Functions of bounded variation and sets of finite perimeter}
\label{subsec_bv}
If $\Om \subseteq \R^d$ is open, we say that $u \in BV(\Om;\R^m)$ if $u \in L^1(\Om;\R^m)$ and its derivative in the sense of distributions is a finite Radon measure on $\Om$, i.e., $Du \in \mathcal{M}_b(\Om;M^{d\times m})$. $BV(\Om;\R^m)$ is called the space of {\it functions of bounded variation} on $\Om$ with values in $\R^m$ and it 
  is a Banach space under the norm $\|u\|_{BV(\Om;\R^m)}:=\|u\|_{L^1(\Om;\R^m)}+\|Du\|_{\mathcal{M}_b(\Om;M^{d\times m})}$.  We call $|Du|(\Om):=\|Du\|_{\mathcal{M}_b(\Om;M^{d\times m})}$ the {\it total variation} of $u$. 
We refer the reader to \cite{AFP} for an exhaustive treatment of the space $BV$.
\par
We say that $u\in SBV(\Om;\R^m)$ if $u\in BV(\Om;\R^m)$ and its distributional derivative can be written in the form
$$
Du=\nabla u\,dx+(u^+-u^-)\otimes \nu_u \Ha^{d-1}\lfloor J_u,
$$
where $\nabla u\in L^1(\Om;M^{d\times m})$ denotes the approximate gradient of $u$, $J_u$ denotes the set of approximate jumps of $u$, $u^+$ and $u^-$ are the traces of $u$ on $J_u$, and $\nu_u(x)$ is the normal to $J_u$ at $x$. 
\par
Note that if $u \in SBV(\Om;\R^m)$, then the singular part of $Du$ is concentrated on $J_u$ which is a countably $\Ha^{d-1}$-rectifiable set: there exists a set $E$ with $\Ha^{d-1}(E)=0$ and a sequence $(M_i)_{i \in \N}$ of $C^1$-submanifolds of $\R^d$ such that $J_u \subseteq E \cup \bigcup_{i \in \N}M_i$.
\par
We will say that $E\subseteq\R^d$ with $|E|<+\infty$ has finite perimeter if $1_E\in BV(\R^d)$. The perimeter of $E$ is defined as
$$
Per(E)=|D1_E|(\R^d).
$$
It turns out that
$$
D1_E=\nu_E \Ha^{d-1}\lfloor \partial^*E,\qquad Per(E)=\Ha^{d-1}(\partial^*E),
$$
where $\partial^*E$ is called the {\it reduced boundary} of $E$, and $\nu_E$ is the associated inner approximate normal (see \cite[Section 3.5]{AFP}). We have that $\partial^*E\subseteq \partial E$, but the topological boundary can in in general be much larger than the reduced one. If $A\subseteq \R^d$ is open and bounded with $\Ha^{d-1}(A)<+\infty$, then $A$ has finite perimeter with $Per(A)\leq \Ha^{d-1}(\partial A)$.

\vskip10pt\noindent
\subsection{Functions of bounded deformation}
If $\Om \subseteq \R^d$ is open, we say that $u \in BD(\Om)$ if $u \in L^1(\Om;\R^d)$ and its symmetric gradient $Eu:=\frac{Du+(Du)^*}{2}$ in the sense of distributions is a finite Radon measure on $\Om$, i.e., $Eu \in \mathcal{M}_b(\Om;{\rm M}^{d\times d}_{sym})$. $BD(\Om)$ is called the space of {\it functions of bounded deformation} on $\Om$. We refer the reader to \cite{Tem-Str, Temam} for the main properties of the space $BD$.
\par
We will make use of a subspace of $BD(\Om)$ called the space of {\it special functions of bounded deformation} introduced in \cite{ACDM}. We say that $u\in SBD(\Om)$ if $u\in BD(\Om)$ and its symmetrized distributional derivative can be written in the form
$$
Eu=e(u)\,dx+(u^+-u^-)\odot \nu_u \Ha^{d-1}\lfloor J_u,
$$
where $e(u)\in L^1(\Om;{\rm M}^{d\times d}_{sym})$ denotes the approximate symmetrized gradient of $u$, $J_u$ denotes the set of approximate jumps of $u$, $u^+$ and $u^-$ are the traces of $u$ on $J_u$, and $\nu_u(x)$ is the normal to $J_u$ at $x$. As in the case of functions of bounded variation, $J_u$ is a $\Ha^{d-1}$-countably rectifiable set.
\par
We will use the following compactness and lower semicontinuity result proved in \cite{BCDM}.

 \begin{theorem}
\label{thm:sbd}
Let $\Om \subseteq \R^d$ be open, bounded and with a Lipschitz boundary, and let $(u_n)_{n\in\N}$ be a sequence in $SBD(\Om)$ such that
$$
\sup_n\left[|Eu_n|(\Om)+\|u_n\|_{L^1(\Om;\R^d)}+\|e(u_n)\|_{L^p(\Om;{\rm M}^{d\times d}_{sym})}+\Ha^{d-1}(J_{u_n})\right]<+\infty
$$
for some $p>1$. Then there exists $u\in SBD(\Om)$ and a subsequence $(u_{n_k})_{k\in \N}$ such that
$$
u_{n_k}\to u\qquad\text{strongly in }L^1(\Om;\R^d),
$$
$$
e(u_{n_k})\rightharpoonup e(u)\qquad\text{weakly in }L^p(\Om;{\rm M}^{d\times d}_{sym}),
$$
and
$$
\Ha^{d-1}(J_u)\leq \liminf_{k\to +\infty} \Ha^{d-1}(J_{u_{n_k}}).
$$
\end{theorem}

We will need also some properties of the {\it sections} of $SBD$-functions. If $\Om\subseteq\R^d$ is open and $u\in SBD(\Om)$, let us consider the scalar function on $\Om^\xi_y$ given by
\begin{equation}\label{eq:uxi}
\hat u^\xi_y(t):=u(y+t\xi)\cdot \xi
\end{equation}
and the set
\begin{equation} \label{eq:Juxi}
J_u^\xi:=\{x\in J_u\,:\, (u^+(x)-u^-(x))\cdot \xi\not=0\}
\end{equation}
The following result holds true (see \cite{ACDM}).

 \begin{theorem}[\bf One dimensional sections of $SBD$-functions]
\label{thm:sbd-1s}
Let $\Om\subseteq \R^d$ be open, $\xi\in \R^d$ with $|\xi|=1$ and let $u\in SBD(\Om)$. Then for $\Ha^{d-1}$-a.e. $y\in \Om^\xi$ we have
$$
\hat u^\xi_y\in SBV(\Om^\xi_y)
$$
with
$$
(\hat u^\xi_y)'(t)=(e(u)\xi\cdot \xi)(y+t\xi)\qquad\text{for a.e. $t\in \Om^\xi_y$}
$$
and
$$
J_{\hat u^\xi_y}=(J^\xi_u)^\xi_y.
$$
\end{theorem}

\section{Obstacles in Stokes fluids and drag minimization}
\label{sec_stokes}
In this section we explain the drag problem for an obstacle immersed in a stationary flow.

\subsection{The flow around the obstacle}
\label{subsec_flow}
Let $\Om\subset \R^d$ be an open bounded set with Lipschitz boundary, and let $V\in C^1(\R^d;\R^d)$ be a divergence free vector field. Given $E\Subset \Om$ open and with a Lipschitz boundary, let us consider the stationary flow for a viscous incompressible fluid around $E$ with boundary conditions on $\partial \Om$ given by $V$, and with Navier boundary conditions on $\partial E$. More precisely, if $u:\Om\setminus E\to \R^d$ is the velocity field, we require that the following items hold true.
\begin{itemize}
\item[(a)] {\it Incompressibility}: ${\rm div}\,u=0$ in $\Om\setminus E$.
\item[(b)] {\it Boundary conditions}: we have
$$
u=V\text{ on $\partial \Om$}\qquad\text{and}\qquad \text{the non-penetration condition } u\cdot \nu=0\text{ on $\partial E$},
$$
where $\nu$ denotes the exterior normal to $E$.
\item[(c)] {\it Equilibrium}: considering the stress 
\begin{equation}
\label{eq_const-sigma}
\sigma:=-p I_d+2\mu e(u),
\end{equation}
where $\mu>0$ is a viscosity parameter, $e(u)$ the symmetrized gradient of $u$ (also denoted by $D(u)$) and $p$ is the pressure, we require
\begin{equation}
\label{eq_div-sigma}
{\rm div}\,\sigma=0\qquad\text{in }\Om\setminus E.
\end{equation}
\item[(d)] {\it Navier conditions on the obstacle}: we have
$$
(\sigma\nu)_\tau=\beta u\qquad\text{on }\partial E,
$$
where $\beta>0$ is a friction parameter, and $(\sigma\nu)_\tau$ denotes the tangential component of force $\sigma\nu$.
\end{itemize}
The stationary flow has the following variational characterization: $u$ is the minimizer of the energy
\begin{equation}
\label{eq_es}
\E(u):=2\mu\int_{\Om\setminus E}|e(u)|^2\,dx+\beta\int_{\partial E}|u|^2\,d\Ha^{d-1}
\end{equation}
among the class of (sufficiently regular) admissible fields
\begin{equation}
\label{eq_def-vub}
\mathcal{V}^{\text{reg}}_{E,V}(\Om):=\{v\in H^1(\Om\setminus E;\R^d)\,:\, \text{$v$ satisfies points (a) and (b)}\},
\end{equation}
where $\Ha^{d-1}$ stands for the $(d-1)$-dimensional Hausdorff measures, which reduces to the area measure on sufficiently regular sets.
Indeed if $u$ is a minimizer, and $\varphi$ is an admissible variation, so that $\varphi=0$ on $\partial \Om$, we get 
\begin{align*}
0&=2\mu\int_{\Om\setminus E}e(u):e(\varphi)\,dx+\beta\int_{\partial E}u\cdot \varphi\,d\Ha^{d-1}\\
&=2\mu\int_{\Om\setminus E}e(u):\nabla \varphi\,dx+\beta\int_{\partial E}u\cdot \varphi\,d\Ha^{d-1}\\
&=-2\mu\int_{\Om\setminus E}{\rm div}\, e(u)\cdot \varphi\,dx+\int_{\partial E}[-2\mu e(u)\nu+\beta u]\cdot \varphi\,d\Ha^{d-1}
\end{align*}
In particular, choosing $\varphi$ with compact support in $\Om\setminus E$ we have
$$
2\mu {\rm div}\, e(u)=\nabla p
$$
for some pressure field $p$: as a consequence $\sigma:=-pI_d+2\mu e(u)$ satisfies \eqref{eq_div-sigma} of condition (c).
\par
Since the admissible functions  $\varphi$ are tangent to $\partial E$, the optimality condition reduces to
\begin{equation}
\label{eq_opt-bdry}
0=\int_{\partial E}[-2\mu e(u)\nu+\beta u]\cdot \varphi\,d\Ha^{d-1}=\int_{\partial E}[-\sigma\nu+\beta u]\cdot \varphi\,d\Ha^{d-1}.
\end{equation}
Notice that every tangential vector field $\eta$ on $\partial E$ can be extended to a divergence free vector field on $\Om\setminus E$ which vanishes on $\partial\Om$, hence it is the trace of an admissible variation $\varphi$: indeed any extension $W$ which vanishes on $\partial \Om$ has a divergence with zero mean, so that considering $W_1$ with ${\rm div}\,W_1={\rm div}\,W$ with $W_1=0$ on $\partial \Om$ and on $\partial E$ (whose existence is guaranteed, for example by \cite[Theorem IV.3.1]{Boyer})), the required extension is given by $W-W_1$. We conclude that the optimality condition \eqref{eq_opt-bdry} yields the Navier condition of point (b).

\subsection{The drag force}
\label{subsec_drag}
Assume now that the external vector field $V$ is equal to a constant $V_\infty\in \R^d\setminus\{0\}$, i.e. the obstacle $E$ is immersed in a uniform flow. The flow is perturbed near $E$, assuming the value $u$, and the obstacle experiences a force which has a component in the direction $V_\infty$ which is given by
$$
Drag(E):=\int_{\partial E}\sigma\nu\cdot \frac{V_\infty}{|V_\infty|}\,d\Ha^{d-1},
$$
which is called the {\it drag force} on $E$ in the direction of the flow.
\par
We claim that
\begin{equation}
\label{eq_drag}
Drag(E)= \frac{1}{|V_\infty|}  \E(u),
\end{equation}
where $\E(u)$ is the energy defined in \eqref{eq_es}. Using the facts that $\sigma$ is symmetric and with zero divergence (so that also the vector field $\sigma V_\infty$ is divergence free), and that $u=V_\infty$ on $\partial \Om$, we may write
 \begin{equation}\label{eq_drag1}
\begin{split}
\int_{\partial E}\sigma\nu\cdot V_\infty\,d\Ha^{d-1}=\int_{\partial E}\sigma V_\infty\cdot \nu\,d\Ha^{d-1}=\int_{\partial \Om}\sigma V_\infty\cdot \nu\,d\Ha^{d-1}=\int_{\partial \Om}\sigma u\cdot \nu\,d\Ha^{d-1}\\
=\int_{\Om\setminus E}{\rm div}\,(\sigma u)\,dx+\int_{\partial E}\sigma u\cdot \nu\,d\Ha^{d-1}=\int_{\Om\setminus E}\sigma:\nabla u\,dx+\int_{\partial E}\sigma \nu\cdot u\,d\Ha^{d-1}.
\end{split}
\end{equation}

Using again that $\sigma$ is symmetric and that $u$ is divergence free, together with the constitutive equation \eqref{eq_const-sigma}, we have
\begin{align*}
\int_{\Om\setminus E}\sigma:\nabla u\,dx&=\int_{\Om\setminus E}\sigma:e(u)\,dx=
\int_{\Om\setminus E}(-p\,I_d+2\mu e(u)): e(u)\,dx\\
&=\int_{\Om\setminus E}(-p\,{\rm div}\,u+2\mu |e(u|^2)\,dx=2\mu\int_{\Om\setminus E}|e(u)|^2\,dx,
\end{align*}
while in view of the Navier conditions on $\partial E$ and the fact that $u$ is tangent to the obstacle
$$
\int_{\partial E}\sigma \nu\cdot u\,d\Ha^{d-1}=\int_{\partial E}(\sigma \nu)_\tau\cdot u\,d\Ha^{d-1}=\beta\int_{\partial E}|u|^2\,d\Ha^{d-1}.
$$
Inserting into \eqref{eq_drag1}, we get that \eqref{eq_drag} follows.

\subsection{The optimization problem}
\label{subsec_opt-pb}
Let $c>0$ and let $f:(0,|\Om|)\to\R\cup\{+\infty\}$ be a lower semicontinuous functions that is not identically equal to $+\infty$. We are concerned with the following optimization problem:
$$
\min_E \left\{Drag(E)+c\Ha^{d-1}(\partial E)+f(|E|)\right\}.
$$
We are thus interested in finding the optimal shape of an obstacle which minimizes the drag force, under a penalization involving its perimeter and its volume.
\par
In view of the energetic characterization of the drag force established in Subsection \ref{subsec_drag}, we can formulate the problem as a minimization problem among the pairs $(E,u)$, where $u$ is a velocity field belonging to the family $\mathcal{V}^{\text{reg}}_{E,V_\infty}(\Om)$ defined in \eqref{eq_def-vub}:
$$
\min_{E,u\in \mathcal{V}^{\text{reg}}_{E,V_\infty}(\Om)} \left\{\frac{2\mu}{|V_\infty|}\int_{\Om\setminus E}|e(u)|^2\,dx+\frac{\beta}{|V_\infty|}\int_{\partial E}|u|^2\,d\Ha^{d-1}+c\Ha^{d-1}(\partial E)+ f(|E|)\right\}.
$$
\vskip15pt
Setting all the constants equal to 1, and  replacing $V_\infty$ by a given divergence free velocity vector field $V$ as in Subsection \ref{subsec_flow}, the drag minimization problem above is a particular case of the following shape optimization problem
\begin{equation}
\label{eq_pb-min}
\min_{E,u\in \mathcal{V}^{\text{reg}}_{E,V}(\Om)} \left\{\int_{\Om\setminus E}|e(u)|^2\,dx+\int_{\partial E}|u|^2\,d\Ha^{d-1}+\Ha^{d-1}(\partial E)+f(|E|)\right\}.
\end{equation}
If we want to apply the direct method of the calculus of variations to the problem, i.e., if we want to recover a minimizer by looking at minimizing sequences $(E_n,u_n)_{n\in\N}$, the following considerations are quite natural.
\begin{itemize}
\item[(a)] Since the problem involves the perimeter of $E$, the sequence $(E_n)_{n\in\N}$ is relatively compact in the family of {\it sets of finite perimeter} (see Section \ref{sec_prel}).
\item[(b)] Concerning the velocities, it turns out naturally that it is convenient to consider also {\it discontinuous} vector fields. Indeed if $u_n\to u$ in some sense, and $\partial E_n$ collapses in some parts generating a surface $\Gamma$ {\it outside} the limit set $E$, the limit velocity field $u$ can present,  in general, discontinuities across $\Gamma$. 
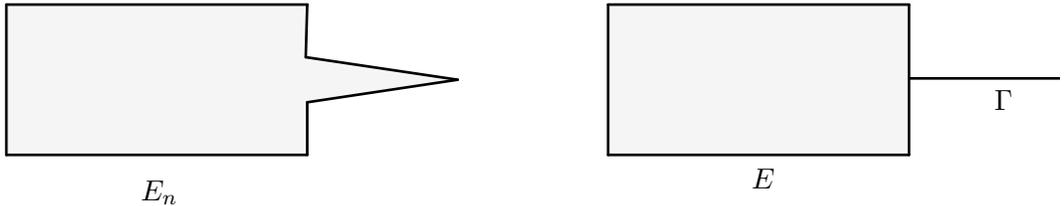
\begin{center}
\definecolor{uuuuuu}{rgb}{0.26666666666666666,0.26666666666666666,0.26666666666666666}
\begin{tikzpicture}[line cap=round,line join=round,>=triangle 45,x=1.0cm,y=1.0cm]
\clip(-4.1,1) rectangle (10,5);
\fill[line width=1.pt,fill=black,fill opacity=0.03999999910593033] (-4.,4.) -- (-4.,2.) -- (0.,2.) -- (0.,2.7) -- (2.,3.) -- (-0.02,3.3) -- (0.,4.) -- cycle;
\fill[line width=1.pt,fill=black,fill opacity=0.03999999910593033] (4.,4.) -- (4.,2.) -- (8.,2.) -- (8.,4.) -- cycle;
\draw [line width=1.pt] (-4.,4.)-- (-4.,2.);
\draw [line width=1.pt] (-4.,2.)-- (0.,2.);
\draw [line width=1.pt] (0.,2.)-- (0.,2.7);
\draw [line width=1.pt] (0.,2.7)-- (2.,3.);
\draw [line width=1.pt] (2.,3.)-- (-0.02,3.3);
\draw [line width=1.pt] (-0.02,3.3)-- (0.,4.);
\draw [line width=1.pt,] (0.,4.)-- (-4.,4.);
\draw [line width=1.pt] (4.,4.)-- (4.,2.);
\draw [line width=1.pt] (4.,2.)-- (8.,2.);
\draw [line width=1.pt] (8.,2.)-- (8.,4.);
\draw [line width=1.pt] (8.,4.)-- (4.,4.);
\draw [line width=1.pt] (8.,3.017863264544199)-- (10.035726529088404,3.017863264544199);
\draw (-2.3427142491303754,1.7938981576175637) node[anchor=north west] {$E_n$};
\draw (5.77353744486048,1.9562231914973804) node[anchor=north west] {$E$};
\draw (9,3) node[anchor=north west] {$\Gamma$};
\end{tikzpicture}
\end{center}
We thus expect an extra term in the surface integral related to the Navier conditions, which amounts at least to
$$
\int_{\Gamma\setminus \partial E}[|u^+|^2+|u^-|^2]\,d\Ha^{d-1},
$$
where $u^\pm$ are the two traces from both sides of $\Gamma$.
\end{itemize}
\par
The previous considerations yield to formulate a relaxed version of problem \eqref{eq_pb-min} in which $E$ varies among the family of sets of finite perimeter contained in $\Om$, while the family of associated admissible velocity fields $u$ is naturally contained in the space of {\it special functions of bounded deformation} $SBD(\Om)$ (see Section \ref{sec_prel}).

\par
In Section \ref{sec_ex-ob}, we will give a precise formulation of problem in this weak setting, which guarantees existence of optimal solutions, describing in particular how the boundary conditions on $\partial\Om$ and on the obstacle have to be rephrased in this context.
\par

\section{A relaxed formulation of the shape optimization problem and statements of the main results}
\label{sec_ex-ob}

Let $\Om\subseteq \R^d$ be open, bounded and with a Lipschitz boundary, and let $V\in C^1(\R^d;\R^d)$ be a divergence free vector field.
In order to deal conveniently with the boundary conditions, let us consider $\Om'\subseteq \R^d$ open and bounded such that $\Om\Subset \Om'$. 
\par
The following definition deals with the family of admissible configurations in the relaxed setting.

\begin{definition}[\bf The class $\as(V)$ of admissible obstacle-velocity configurations]
\label{def:adm-Eu}
We say  
that $(E,u)$ is an admissible configuration for the external velocity $V$, and we will write $(E,u)\in \as(V)$, if $E\subseteq \Om$ is a set of finite perimeter, while
$$
u\in SBD(\Om')\cap L^2(\Om';\R^d)
$$ 
is such that $u=0$ a.e. on $E$ and the following conditions are satisfied.
\begin{itemize}
\item[(a)] The flow is divergence free: ${\rm div}\, u=0 \text{ in the sense of distributions in $\Om'$}$.
\item[(b)] External boundary conditions: $u=V \text{ a.e. on }\Om'\setminus \Omb$.
\item[(c)] Non-penetration condition on the obstacle:
$$
u^{\pm}\cdot \nu=0 \text{ on }\partial^* E\cup J_u,
$$
where $\nu$ denotes the normal to the rectifiable set $\partial^*E\cup J_u$.
\end{itemize}
\end{definition}

\begin{remark}
\label{difference}
{\rm
The crucial difference between admissible velocities in the present framework and those of the family $\mathcal{V}_{E,V}^\text{reg}(\Om)$ introduced before (see \eqref{eq_def-vub}) is that they may have discontinuities outside of $E$.  Within the new setting, the global obstacle is given by 
$$
E\cup J_u
$$ 
i.e. it may contain $(d-1)$ dimensional parts.
\par
Given $(E,u)\in \as(V)$, concerning the traces of $u$ on $\partial^*E$, we will denote with $u^+$ the trace in the direction of the external normal $\nu_E$, so that $u^-=0$ $\Ha^{d-1}$-a.e. on $\partial^*E$. 
\par
Concerning the non-penetration constraint, notice that it suffices to require it only on $J_u$, since it is then automatically verified also on $\partial^*E$. Indeed for $\Ha^{d-1}$-a.e. $x\in \partial^*E\setminus J_u$, we have $u^-(x)=u^+(x)=0$ and the constraint is verified, while for $\Ha^{d-1}$-a.e. $x\in J_u\cap \partial^*E$ the two rectifiable sets $J_u$ and $\partial^*E$ share the same normal vector.
}
\end{remark}

\begin{remark}
\label{rem:sbd}
{\rm
The space $SBD(\Om')$ is naturally a subspace of $L^1(\Om';\R^d)$: we require for admissibility that $u\in L^2(\Om';\R^d)$ to ensure that the velocity field has finite kinetic energy.
It will turn out that velocities in $SBD(\Om')$ which are interesting for our problem (i.e., with finite energy) are automatically elements of $L^2(\Om';\R^d)$ (see Theorem \ref{thm:embedding}).
}
\end{remark}

\begin{remark}[\bf On the boundary condition]
\label{rem:bdry}
{\rm
If $(E,u)\in \as(V)$, then $u\in SBD(\Om')$ with $u=V$ a.e. on $\Om'\setminus \Omb$, so that
$$
J_u\cap \partial \Om=\{x\in \partial \Om\,:\, \gamma(u)(x)\not=V(x)\},
$$
where $\gamma(u)$ is the trace of $u$ on $\partial\Om$ coming from $\Om$ (i.e., the usual trace of $u$ seen as an element of $SBD(\Om)$).
We conclude that within the present framework, the boundary condition is somehow relaxed: a possible mismatch between $u$ and $V$ on $\partial\Om$ is admitted, but then the zone is counted as a jump part of the velocity field, and consequently as a part of the obstacle $\partial^*E\cup J_u$, and will carry a contribution for the energy (see \eqref{eq_JEu} below). Such a relaxation of the boundary condition is a feature which is common to several applications of functions of bounded variation to problems in continuum mechanics (see for example \cite{FL, DMFT} in connection to fracture mechanics or \cite{DMDEM} for problems in plasticity).
}
\end{remark}

\begin{remark}
{\rm
 Given $(E,u)\in \as(V)$,  the obstacle  $E\cup J_u$  may touch $\partial\Om$ only on those part where $V$ is tangent to $\Om$: this is due to the fact that on $(\partial^*E \cup J_u)\cap \partial\Om$, the two sets share $\Ha^{d-1}$-a.e. the same normal, and $u^+=V$  (if the orientation is suitably chosen). 
}
\end{remark}

\begin{remark}
\label{rem:nonempty}
{\rm
Let $E\Subset \Om$ be open and with a Lipschitz boundary. Then we can find $W\in H^1(\Om\setminus E;\R^d)$ such that $W=V$ on $\partial\Om$, $W=0$ on $\partial E$ and $\mbox{div}\, W=0$. Indeed if $\varphi \in C^\infty(\R^d)$ is such that $\varphi=1$ on a neighborhood of $\R^d\setminus \Om$ and $\varphi=0$ on a neighborhood of $E$, we can consider the vector field $V_1:=\varphi V$, whose divergence has zero mean on $\Om\setminus E$ (by Gauss theorem). Then we can find $V_2\in H^1_0(\Om\setminus E;\R^d)$ such that $\mbox{div}\,V=\mbox{div}\, V_1$ (see \cite[Theorem IV.3.1]{Boyer}), so that the field $W:=V_1-V_2$ is an admissible choice. In particular we get that $(E,W)\in \as(V)$, so that the class of admissible configurations is not empty.
}
\end{remark}

Let 
\begin{equation}
\label{eq:f}
\text{$f:[0,|\Om|]\to [0,+\infty]$ be lower semicontinuous, not identically equal to $+\infty$.}
\end{equation}
For every $(E,u)\in \as(V)$, let us set (normalizing to $1$ the constants involved in the drag force problem)
\begin{equation}
\label{eq_JEu}
\begin{split}
\J(E,u):=&\int_{ \Om'} |e(u)|^2\,dx+\int_{\partial^*E}|u^+|^2\,d\Ha^{d-1}+\int_{J_u\setminus \partial^*E}[|u^+|^2+|u^-|^2]\,d\Ha^{d-1}
\\
&+\Ha^{d-1}(\partial^*E)+2\Ha^{d-1}(J_u\setminus \partial^*E)+f(|E|).
\end{split}
\end{equation}

\begin{remark}
\label{rem:functional}
{\rm

Concerning the volume integral in $\J(E,u)$, the density $e(u)$ is equal to $e(V)$ a.e. on $\Om'\setminus \Omb$ and equal to $0$ a.e. on $E$: as a consequence we could replace it with an integral on $\Om\setminus E$ without affecting the minimization of $\J$. \par
Concerning the {\it Navier energy} and the surface penalization for $\partial^*E\cup J_u$, notice that it counts also for the possible mismatch at the boundary between $u$ and $V$ as pointed out in Remark \ref{rem:bdry}: the mismatch is thus ``penalized''
by the energy of the problem.
\par
The previous observations show that the larger domain $\Om'$ plays only an instrumental role for the problem, as it can be replaced by any open domain strictly containing $\Om$.

}
\end{remark}

 The first main result of the paper is the following

\begin{theorem}[\bf Existence of optimal obstacles]\label{main1}
Let $\Om\subseteq\R^d$ be a bounded open set with Lipschitz boundary, $V\in C^1(\R^d;\R^d)$ a divergence-free vector field, and $f$ a function satisfying \eqref{eq:f}. Let the family of admissible configurations $\as(V)$ be given by Definition \ref{def:adm-Eu} and let $\J$ be the functional defined in \eqref{eq_JEu}. Then the problem 
\begin{equation}\label{mainpb}
\min_{(E,u)\in\as(V)}\J(E,u)
\end{equation}
admits a solution.
\end{theorem}

\begin{remark}
\label{rem:original}
{\rm

We recover the original drag minimization problem when $V$ is a constant nonzero vector $V_\infty$, and we restore properly in the functional the physical constants $\mu$ and $\beta$, together with the perimeter penalization constant $c$.

}
\end{remark}

 The second main result of the paper concerns the regularity of minimizers in the two dimensional setting.

 \begin{theorem}[\bf Regularity in dimension two]
 \label{main2}
Let $\Om\subseteq\R^2$ be a bounded open set with Lipschitz boundary, $V\in C^1(\R^2;\R^2)$ a divergence-free vector field, and $f:[0,|\Om|]\to [0,+\infty[$ a Lipschitz function. Let $(E,u) \in \as(V)$ be a solution to \eqref{mainpb} according to Theorem \ref{main1}. Then 
\[
\Ha^1\left(\Om\cap(\overline{J_u\cup\partial^*E}\setminus(J_u\cup\partial^*E))\right)=0,
\]
and $u\in C^\infty(\Om \setminus \overline{J_u\cup\partial^*E};\R^2)$.
\end{theorem}
 Theorem \ref{main1} will be proved in Section \ref{sec_proof-main}, on the basis of some technical results established in \ref{sec_sbd-tech}. The proof of Theorem \ref{main2} will be addressed in Section \ref{sec-reg}.

\section{Some technical results in SBD}
\label{sec_sbd-tech}
In this section we collect some technical properties concerning the space SBD that will be fundamental in the proof of Theorem \ref{main1}. In particular in Theorem \ref{thm:embedding} we will prove that admissible velocity vector fields enjoy higher summability properties (indeed they belong to $L^{\frac{2d}{d-1}}$). In 
Theorem \ref{thm:closure} we will prove that velocity fields $u$ with $u^\pm$ tangent to the discontinuity set $J_u$ form a closed set under the natural convergence of minimizing sequences for the main optimization problem. Finally in Theorem \ref{thm:lsc-phi} we will prove a lower semicontinuity result for surface energies depending on the traces, which entails in particular the lower semicontinuity of the term associated to the Navier conditions.

\subsection{An immersion result}
The following embedding result holds true.

 \begin{theorem}
\label{thm:embedding}
Let $\Om\subseteq \R^d$ be a bounded open set, and let $u\in SBD(\R^d)$ be supported in $\Om$ such that
$$
\E(u):=\int_\Om |e(u)|^2\,dx+\int_{J_u}\left[|u^+|^2+|u^-|^2\right]\,d\Ha^{d-1}<+\infty.
$$
Then $u\in L^{\frac{2d}{d-1}}(\Om)$ with
$$
\|u\|_{\frac{2d}{d-1}}\leq C \sqrt{\E(u)},
$$
where $C$ depends on $d$ and $\text{diam}(\Om)$ only.
\end{theorem}

\begin{proof}
It suffices to follow the strategy of the proof of the classical embedding of $BD$ into $L^{d/d-1}$ explained in \cite{Temam}, but concentrating on the square of the components.
\par
Let us consider the unit vector
$$
\xi:=\frac{1}{\sqrt{d}}(1,1,\dots,1)\in \R^d.
$$
Employing the characterization by sections recalled in Section \ref{sec_prel}, for $\Ha^{d-1}$-a.e. $y\in \xi^\perp$ we have
$$
\hat u^\xi_y \in SBV(\Om^\xi_y)
$$
with
$$
\int_{\Om^\xi_y} |(\hat u^\xi_y)'|^2\,dt+\sum_{t\in J_{\hat u^\xi_y}}\left[ |(\hat u^\xi_y)^+(t)|^2+|(\hat u^\xi_y)^-(t)|^2\right]<+\infty.
$$
Then we can write for a.e. $t\in \R$
 \begin{equation}\label{eq_gxi}
\begin{split}
\Vert\hat u^\xi_y\Vert_{L^\infty(\Om_y^\xi)}^2&\leq \left|D(\hat u^\xi_y)^2\right|(\Om^\xi_y)=\int_{\Om_y^\xi}2|\hat{u}_y^\xi (\hat{u}_y^\xi)'|dt+\sum_{t\in J_{\hat{u}_y^\xi}}\left||(\hat{u}_y^{\xi})^+(t)|^2-|(\hat{u}_y^{\xi})^-(t)|^2\right|\\
&\leq \frac{1}{2}\Vert \hat{u}_y^\xi\Vert_{L^\infty(\Om_y^\xi)}^2+2|\Om_y^\xi|\int_{\Om_y^\xi}\left|(\hat{u}_y^\xi)'\right|^2dt+\sum_{t\in J_{\hat{u}_y^\xi}}\left(\left|(\hat{u}_y^{\xi})^+(t)\right|^2+\left|(\hat{u}_y^{\xi})^-(t)\right|^2\right),
\end{split}
\end{equation}
Let us set
$$
g_\xi(x):=\int_{\Om^\xi_y}|(\hat u^\xi_y)'|^2\,dt+\sum_{t\in J_{\hat u^\xi_y}}\left[ |(\hat u^\xi_y)^+(t)|^2+|(\hat u^\xi_y)^-(t)|^2\right],
$$
where $y:=\pi_{\xi^\perp}(x)$, i.e., the projection of $x$ on the hyperplane $\xi^\perp$. $g_\xi(x)$ only depends on the projection of $x$ on $\xi^\bot$ and
\begin{align*}
\int_{\xi^\bot}g_\xi d\Ha^{d-1}&= \int_\Om |e(u)\xi\cdot \xi|^2\,dx+\int_{J_u}\left[|u^+|^2+|u^-|^2\right]|\xi\cdot \nu|\,d\Ha^{d-1}\\
&\leq C\left[\int_\Om |e(u)|^2\,dx+\int_{J_u}\left[|u^+|^2+|u^-|^2\right]\,d\Ha^{d-1}\right]
\end{align*}
 
where $C$ depends only on $d$. Thanks to \eqref{eq_gxi} we have
\begin{equation}
\label{eq_ineq-2}
|\xi\cdot u|^2\leq C g_\xi\qquad\text{a.e. on }\Om,
\end{equation}
where $C$ depends on the diameter of $\Om$, and from now on all the constants $C$ that appear depend on $n,\text{diam}(\Om)$. For every $k=1,\dots,d-1$, we can write
$$
\xi=\frac{1}{\sqrt{d}}e_k+\sqrt{\frac{d-1}{d}}h_k,
$$
where $e_k$ is the $k$-th vector of the canonical base, and $h_k$ is the unit vector in the direction $\sqrt{d}\xi-e_k$. Reasoning as above on the decomposition
$$
\xi\cdot u=\sqrt{\frac{d-1}{d}}h_k\cdot u+\frac{1}{\sqrt{d}}e_k\cdot u
$$
we obtain a similar estimate
\begin{equation}
\label{eq_ineq-2bis}
|\xi \cdot u|^2\leq C\left( g_{h_k}+g_{e_k}\right),
\end{equation}
Multiplying inequality \eqref{eq_ineq-2} with inequalities \eqref{eq_ineq-2bis} for $k=1,\dots,d-1$, we obtain reasoning as in \cite[Chapter II, Theorem 1.2]{Temam}
$$
\|(\xi\cdot u)^2\|_{\frac{d}{d-1}}\leq C\left[ \int_\Om |e(u)|^2\,dx+\int_{J_u}\left[|u^+|^2+|u^-|^2\right]\,d\Ha^{d-1}\right].
$$
{Since this estimate does not depend on the particular choice of the
  basis and hence holds for any $\xi$ with norm one, the theorem is proved.} 
\end{proof}

\subsection{Closure of the non-penetration constraint} 

In the context of equi-Lipschitz boundaries, the preservation of the non-penetration property for a sequence of Sobolev functions converging weakly, comes rather directly via the divergence theorem (we refer the reader, for instance, to \cite{BFN10}). However, in the case of collapsing boundaries, so that the limit function lives on both sides of a surface and in absence of any smoothness of the limit set, this technique does not work. The proof of the non-penetration preservation requires different technical arguments that we handle in the SBD context.

Let us start with the following lower semicontinuity result.

 \begin{theorem}
\label{thm:sci-normal}
Let $\Om\subseteq \R^d$ be a bounded open set, and let $(u_n)_{n\in\N}$ be a sequence in $SBD(\Om)$ such that
$$
\sup_n \left[\int_\Om |e(u_n)|^2\,dx+\Ha^{d-1}(J_{u_n})\right]<+\infty
$$
with
$$
u_n\to u\qquad\text{in measure }
$$
for some $u\in SBD(\Om)$. Then
$$
\int_{J_u}\left[ |u^+\cdot \nu_u|+|u^-\cdot \nu_u|\right]\,d\Ha^{d-1}\leq \liminf_{n\to+\infty}
\int_{J_{u_n}}\left[ |u^+_n\cdot \nu_{u_n}|+|u^-\cdot \nu_{u_n}|\right]\,d\Ha^{d-1}.
$$
\end{theorem}

\begin{proof}
Let us consider a countable set of functions $\{\varphi_h\,:\,h\in \N\}$ which is dense with respect to $\|\cdot \|_\infty$ inside the set
$$
\left\{f\in C^0_c(]0,+\infty[)\,:\, \int_0^{+\infty} f\,\dd t=0 \text{ and }\|f\|_\infty\leq 1\right\}.
$$
Given $\eps>0$, let us consider
$$
g_{h,k}(x):=\int_0^{\frac{1}{2}|x-x_k|^2}\varphi_h(t)\,dt,
$$
where $\{x_k\,:\, k\in\N\}$ is a countable and dense set in $B_\eps(0)\subset \R^d$ with $x_0=0$. Clearly $g_{h,k}\in C^1_c(\R^d)$ with
$$
G_{h,k}(x):=\nabla g_{h,k}(x)=\varphi_h\left( \frac{1}{2}|x-x_k|^2\right) (x-x_k).
$$
We have that $G_{h,k}$ is a continuous conservative vector field with compact support on $\R^d$.
\par
Let us set for $(i,j)\in \R^d\times \R^d$ and $\nu \in \R^d$ with $|\nu|=1$
$$
f_\eps(i,j,\nu):=\sup_{h,k}(G_{h,k}(i)-G_{h,k}(j))\cdot \nu.
$$
By construction $f_\eps$ is a {\it symmetric jointly convex function} according to \cite[Definition 3.1]{FPS}. We claim that for $i\not=j$
\begin{equation}
\label{eq_fe}
|i\cdot \nu|+|j\cdot \nu|\leq f_\eps(i,j,\nu)\leq |i\cdot \nu|+|j\cdot \nu|+2\eps.
\end{equation}
In view of the lower semicontinuity result \cite[Theorem 5.1]{FPS} we have
$$
\liminf_{n\to +\infty} \int_{J_{u_n}}f_\eps(u^+_n,u^-_n,\nu_{u_n})\,d\Ha^{d-1}\geq \int_{J_{u}}f_\eps(u^+,u^-,\nu_{u})\,d\Ha^{d-1}.
$$
We can thus write

\begin{align*}
&\liminf_{n\to+\infty} \left[ \int_{J_{u_n}}\left[ |u^+_n\cdot \nu_{u_n}|+|u^-_n\cdot \nu_{u_n}|\right]\,d\Ha^{d-1}+2\eps \Ha^{d-1}(J_{u_n})\right]
\\
&\geq \liminf_{n\to+\infty} \int_{J_{u_n}}f_\eps(u^+_n,u^-_n,\nu_{u_n})\,d\Ha^{d-1}
\geq \int_{J_{u}}f_\eps(u^+,u^-,\nu_{u})\,d\Ha^{d-1}\\
&\ge
\int_{J_{u}}\left[ |u^+\cdot \nu_{u}|+|u^-\cdot \nu_{u}|\right]\,d\Ha^{d-1},
\end{align*}
so that the result follows taking into account the bound on $\Ha^{d-1}(J_{u_n})$ and letting $\eps\to 0$.
\par
In order to complete the proof, we need to show claim \eqref{eq_fe}. The estimate from above follows from
$$
[G_{h,k}(i)-G_{h,k}(j)]\cdot \nu\leq |(i-x_k)\cdot \nu|+|(j-x_k)\cdot \nu|\leq |i\cdot \nu|+|j\cdot \nu|+2\eps
$$
since  $\|\varphi_h\|_\infty \leq 1$ and $|x_k|<\eps$. Let us prove the estimate from below. We select $x_{k_n}\to 0$ such that $|i-x_{k_n}|\not=|j-x_{k_n}|$ (which is always possibile in view of the density of $\{x_k\,:\, k\in \N\}$ inside $B_\eps(0)$ and since $i\not=j$) and then $\varphi_{h_n}$ such that for $n\to+\infty$
$$
\varphi_{h_n}\left(\frac{1}{2}|i-x_{k_n}|^2\right)\to \frac{i\cdot \nu}{|i\cdot \nu|+\eta}\qquad \text{and}\qquad \varphi_{h_n}\left(\frac{1}{2}|j-x_{k_n}|^2\right)\to 
-\frac{j\cdot \nu}{|j\cdot \nu|+\eta},
$$
where $\eta>0$. By definition of $f_\eps$ we infer that
$$
f_\eps(i,j,\nu)\geq |i\cdot \nu|+|j\cdot \nu|-2\eta,
$$
so that the estimate from below follows by sending $\eta\to 0$.
\end{proof}

We are now in a position to prove the main result of the section.

 \begin{theorem}[\bf Closure of the non-penetration constraint on the jump set]
\label{thm:closure}
Let $\Om\subseteq \R^d$ be a bounded open set, and let $(u_n)_{n\in \N}$ be a sequence in $SBD(\Om)$ such that
$$
\sup_n \left[\int_\Om |e(u_n)|^2\,dx+\Ha^{d-1}(J_{u_n})\right]<+\infty
$$
and
$$
u_n\to u\qquad\text{in measure }
$$
for some $u\in SBD(\Om)$. If
$$
u^\pm_n\cdot \nu_{u_n}=0\qquad\text{$\Ha^{d-1}$-a.e. on }J_{u_n},
$$
then
$$
u^\pm\cdot \nu_{u}=0\qquad\text{$\Ha^{d-1}$-a.e. on }J_{u}.
$$
\end{theorem}

\begin{proof}
By Theorem \ref{thm:sci-normal} we may write
$$
\int_{J_u}\left[ |u^+\cdot \nu_u|+|u^-\cdot \nu_u|\right]\,d\Ha^{d-1}\leq \liminf_{n\to+\infty}
\int_{J_{u_n}}\left[ |u^+_n\cdot \nu_{u_n}|+|u^-\cdot \nu_{u_n}|\right]\,d\Ha^{d-1}=0,
$$
so that the result follows.
\end{proof}

\subsection{A lower semicontinuity result for surface energies in $SBD$} 
\label{subsec_lsc-upm}
In this section we deal with the lower semicontinuity of the surface term of the functional $J$ in \eqref{eq_JEu} connected with the Navier conditions on the obstacle. The following lower semicontinuity result holds true.

 \begin{theorem}
\label{thm:lsc-phi}
Let $\Om\subseteq \R^d$ be an open set, $u_n,u \in SBD(\Om)$ such that
$$
u_n \to u\qquad\text{strongly in }L^1(\Om;\R^d)
$$
and
$$
\sup_n \left[ \int_\Om |e(u_n)|^2\,dx+\Ha^{d-1}(J_{u_n})\right]<+\infty.
$$
Then if $\phi:\R^d \to [0,+\infty]$ is a    lower semicontinuous function, we have
$$
\int_{J_u}[\phi(u^+)+\phi(u^-)]\,d\Ha^{d-1}\leq \liminf_{n\to +\infty} 
\int_{J_{u_n}}[\phi(u_n^+)+\phi(u_n^-)]\,d\Ha^{d-1}.
$$
\end{theorem}

This applies in particular to $\phi(u)=|u|^2$ and $\phi(u)=1_{\{u\neq 0\}}$, which will be of interest to us. 

\begin{proof}
Notice first that $\phi$ may be supposed to be continuous. Indeed for any lower-semicontinuous nonnegative $\phi$, by considering a sequence of continuous nonnegative functions $\phi_k\nearrow\phi$ we get
\begin{align*}
\int_{J_u}[\phi(u^+)+\phi(u^-)]\,d\Ha^{d-1}&=\liminf_{k\to\infty}\int_{J_u}[\phi_k(u^+)+\phi_k(u^-)]\,d\Ha^{d-1}\\
&\leq \liminf_{k\to\infty}\liminf_{n\to +\infty} 
\int_{J_{u_n}}[\phi_k(u_n^+)+\phi_k(u_n^-)]\,d\Ha^{d-1}\\
&\leq \liminf_{n\to +\infty} 
\int_{J_{u_n}}[\phi(u_n^+)+\phi(u_n^-)]\,d\Ha^{d-1}
\end{align*}

Through a  by now standard  blow-up argument ( see Remark \ref{rem:blow}), we can reduce the problem to the following lower semicontinuity result. Let $Q_1\subseteq \R^d$ be the unit square centred at $0$, and let us set
$$
H:=Q_1 \cap \{x_d=0\}\qquad\text{and}\qquad Q_1^\pm:=Q_1\cap \{x_d \gtrless 0\}.
$$
Given $u^\pm \in \R^d$ with $u^+\not=u^-$ and $u_n\in SBD(Q_1)$ with
\begin{equation}
\label{eq_unu}
u_n\to u:=u^+1_{Q_1^+}+u^-1_{Q_1^-}\qquad\text{strongly in }L^1(Q_1;\R^d),
\end{equation}
\begin{equation}
\label{eq_jun}
\sup_n \Ha^{d-1}(J_{u_n})<+\infty
\end{equation}
and
\begin{equation}
\label{eq_eun}
e(u_n)\to 0\qquad\text{strongly in }L^1(Q_1;M^{d\times d}_{sym}),
\end{equation}
then 
\begin{equation}
\label{eq:blow-sci}
\phi(u^+)+\phi(u^-)\leq \liminf_{n\to +\infty} \int_{J_{u_n}}[\phi(u_n^+)+\phi(u_n^-)]\,d\Ha^{d-1}.
\end{equation}
We now divide the proof in several steps, and we will employ the characterization by sections of $SBD$ functions explained in Section \ref{sec_prel}.

\vskip10pt\noindent{\bf Step 1.} Let $\eps>0$ be given. We fix $\delta>0$ and $N\in \N$ with $N>d$: these numbers will be subject to several constraints that will appear during the proof. 
\par
Let us fix $N$ unit vectors $\{\xi_i\}_{1\leq i \leq N}$ such that
\begin{equation}
\label{eq_ed-delta}
|e_d\cdot \xi_i-1|<\delta
\end{equation}
and such that any subset of $d$ of them forms a basis of $\R^d$. Moreover, we may assume in addition that 
\begin{equation}
\label{eq_disc-u}
(u^+-u^-)\cdot \xi_i\not=0
\end{equation}
for every $i=1,\dots,N$.
\par
Thanks to \eqref{eq_unu} and \eqref{eq_jun}, we can fix $a>0$ small such that setting $H^\pm:=H \times \{\pm a\}=H\pm a e_d$, we have
$$
(u_n)_{|H^\pm} \to u^\pm \qquad\text{strongly in }L^1(H^\pm;\R^d)
$$
and
\begin{equation}
\label{eq_hdun0}
\forall n\in \N\,:\, \Ha^{d-1}(J_{u_n}\cap H^\pm)=0.
\end{equation}

\vskip10pt\noindent{\bf Step 2.}
We claim that, up to a subsequence, we can find $H^-_\eps\subset H^-$ with
\begin{equation}
\label{eq_1}
\Ha^{d-1}(H^-\setminus H^-_\eps)<\eps
\end{equation}
such that for every $i=1,\dots,N$, for every $y\in H^-_\eps$ and for every $n\in \N$
\begin{equation}
\label{eq_3}
H^-_\eps\cap J_{u_n}=\emptyset,
\end{equation}
and
\begin{equation}
  \label{eq_6bis}
   \Ha^0((J_{u_n})_y^{\xi_i})<+\infty, \quad 
  \Ha^0((J_{u_n})_y^{\xi_i}{ \cap \R_+}) \ge 1 \,.
\end{equation}
Moreover setting
$$
\widehat{(u_n)}_y^{\xi_i}:=u_n(y+t\xi_i)\cdot \xi_i,
$$
for every $y\in H_\eps^-$ we have
\begin{equation*}
\label{eq_bv-xi}
\widehat{(u_n)}_y^{\xi_i}\in SBV((Q_1)^{\xi_i}_y), 
\end{equation*}
\begin{equation}
\label{eq_jump-sec}
 J_{\widehat{(u_n)}_y^{\xi_i}}=\left( J_{u_n}\right)^{\xi_i}_y
\end{equation}
{ (\textit{cf}~notation~\eqref{eq:Juxi}),}
\begin{equation}
\label{eq_oscun1}
\|[\widehat{(u_n)}_y^{\xi_i}]'\|_{L^1}\to 0 \qquad \text{uniformly for }y\in H^-_\eps,
\end{equation}
and
\begin{equation}
\label{eq_un-unif-D2}
(u_n)_{|H^-}\to u^-\qquad\text{uniformly on }H^-_\eps.
\end{equation}
Indeed, if the number $\delta$ appearing in \eqref{eq_ed-delta} is small enough, we can find $A^-_\eps\subseteq H^-$ with
\begin{equation}
\label{eq_Ae}
\Ha^{d-1}(H^-\setminus A^-_\eps)<\frac{\eps}{2}
\end{equation}
and such that for every $y\in A^-_\eps$ the lines $\{y+t\xi_i\,:\, t\in\R\}$ intersect $H^+$ for every $i=1,\dots,N$. In view of \eqref{eq_unu}, \eqref{eq_jun} and \eqref{eq_eun}, and since pointwise convergence implies almost uniform convergence, we can find $N_\eps\subset A^-_\eps$ with 
\begin{equation}
\label{eq_Ne}
\Ha^{d-1}(N_\eps)<\frac{\eps}{2}
\end{equation}
and such that, up to a subsequence
\begin{equation}
\label{eq_unxi-unif}
\|\widehat{(u_n)}_y^{\xi_i}-\widehat{u}_y^{\xi_i}\|_{L^1} \to 0 \qquad \text{uniformly for }y\in A^-_\eps\setminus N_\eps
\end{equation}
\begin{equation}
\label{eq_oscun}
\|[\widehat{(u_n)}_y^{\xi_i}]'\|_{L^1}\to 0 \qquad \text{uniformly for }y\in A^-_\eps\setminus N_\eps
\end{equation}
\begin{equation}
\label{eq_un-unif-D}
(u_n)_{|H^-}\to u^-\qquad\text{uniformly on }A^-_\eps\setminus N_\eps,
\end{equation}
and for every $y\in A^-_\eps\setminus N_\eps$
\begin{equation}
\label{eq_h0fin}
\Ha^0((J_{u_n})^{\xi_i}_y)<+\infty.
\end{equation}
Notice that for $n$ large enough and for every $y\in A^-_\eps\setminus N_\eps$ we have 
\begin{equation}
\label{eq_h01}
(J_{u_n})_y^{\xi_i}\not=\emptyset.
\end{equation}
Indeed otherwise, we would get for $n_k\to +\infty$ the existence of $y_k\in A^-_\eps\setminus N_\eps$ with  $\widehat{(u_{n_k})}_{y_k}^{\xi_i}\in W^{1,1}((Q_1)^{\xi_i}_{y_k})$, and \eqref{eq_un-unif-D} together with \eqref{eq_oscun} would yield
$$
\|\widehat{(u_{n_k})}_{y_k}^{\xi_i}-u^-\|_1 \to 0
$$
against \eqref{eq_unxi-unif} (recall that by the choice \eqref{eq_disc-u} of the $\xi_i$, the functions $\widehat{u}_y^{\xi_i}$ have a jump). 
The claim follows by setting
$$
H^-_\eps:=A_\eps\setminus \left[N_\eps \cup \bigcup_n (J_{u_n}\cap H^-)\right].
$$
Indeed \eqref{eq_1} follows from \eqref{eq_Ae}, \eqref{eq_Ne} and \eqref{eq_hdun0}, while \eqref{eq_3} is clearly satisfied. Relation \eqref{eq_6bis} follows by \eqref{eq_h0fin} and \eqref{eq_h01}, while relation \eqref{eq_oscun1} follows from \eqref{eq_oscun}. Finally relation \eqref{eq_un-unif-D2} follows from \eqref{eq_un-unif-D}.

\vskip10pt\noindent{\bf Step 3.} For every $i=1,\dots,N$, let us consider the set
$J_n^{i,-}$ given by the first point of intersection (with $t>0$) of the line $\{y+t\xi^i\,:\, t\in\R\}$ with the jump set $J_{u_n}$ as $y$ varies in the set $H^-_\eps$ defined in Step 2 (recall \eqref{eq_6bis} and \eqref{eq_jump-sec}). In view of \eqref{eq_oscun1} and \eqref{eq_un-unif-D2}, we can find $\eta_n\to 0$ such that for every $x\in J_n^{i,-}$
with $\nu_{u_n}\cdot \xi_i>0$
\begin{equation}
\label{eq_components}
|u^-_n(x)\cdot \xi_i-u^-\cdot \xi_i|<\eta_n.
\end{equation}

\vskip10pt\noindent{\bf Step 4.} We claim that, for $\delta$ small enough and $N$ large enough, up to a subsequence, we can find $\tilde J_n^- \subseteq J_{u_n}$ with
\begin{equation}
\label{eq_t1}
\Ha^{d-1}(\tilde J_n^-) \ge1-c_\eps,
\end{equation}
where $c_\eps\to 0$ as $\eps\to 0$, and such that for every $x\in \tilde J_n^-$
\begin{equation}
\label{eq_interJ}
x\in J^{i,-}_n \text{ for $d$ different indices $i\in \{1,\dots,N\}$},
\end{equation}
where $J^{i,-}_n$ is defined in Step 3.
Moreover, we can orient $\nu_{u_n}$ on $\tilde J_n^-$ in such a way that
\begin{equation}
\label{eq_orientation}
e_d\cdot \nu_{u_n}>0\qquad\text{and}\qquad \xi_i\cdot \nu_{u_n}>0 \text{ for every $i=1,\dots,N$}.
\end{equation}
Intuitively speaking, the points in $\tilde J^-_n$ are seen from $H^-_\eps$ under $d$ different directions: moreover the associated lines cut the jump transversaly, from the ``lower'' to the ``upper'' part.
\vskip10pt
Indeed, in view of the definition of $\xi_i$ (which form a very small angle with $e_d$ as $\delta\to 0$) and of the area formula {(\textit{cf} for instance~\cite[Sec.~3.2]{Fed})}, we can assume that $\delta$ is so small that for every $i=1,\dots,N$ 
\begin{equation}
\label{eq_area}
\Ha^{d-1}(J_n^{i,-})\geq \int_{J_n^{i,-}}|\nu_{u_n}\cdot \xi_i|\,d\Ha^{d-1}=\Ha^{d-1}((H_\eps^-)^{\xi_i})=
\frac{1}{1+\hat c_\delta}\Ha^{d-1}(H_\eps^-),
\end{equation}
{ where the notation $(H_\eps^-)^{\xi_i}$ is defined in~\eqref{eq:Exi} and}
where $\hat c_\delta\to 0$, so that, taking into account \eqref{eq_1}, for small $\delta$ we have
\begin{equation}
\label{eq_t1tris}
\Ha^{d-1}(J_n^{i,-})\geq 1-2\eps.
\end{equation}
By Lemma \ref{lem:cov} below (with $X=J_{u_n}$, $\mu=\Ha^{d-1}$, and $\mathcal{M}$ given by the family of Borel sets) if $N$ is large enough we can find an index $\bar i$ such that
\begin{equation}
\label{eq_LemmaN}
\Ha^{d-1}\left( J_n^{\bar i,-}\setminus \bigcup_{\stackrel{i_1<i_2<\dots<i_d}{i_h=1,\dots,N}} \left(J_n^{i_1,-}\cap J_n^{i_2,-}\cap\dots\cap J_n^{i_d,-}\right)\right)<\eps.
\end{equation}
Intuitively speaking, most of the points in $J_{n}^{\bar i,-}$ are seen from $H^-_\eps$ at least under $d$ different directions: we call this set $\tilde J_n^-$, i.e.,
\begin{equation}
\label{eq_defJm}
\tilde J_n^-:=J_n^{\bar i,-}\cap \bigcup_{\stackrel{i_1<i_2<\dots<i_d}{i_h=1,\dots,N}} \left(J_n^{i_1,-}\cap J_n^{i_2,-}\cap\dots\cap J_n^{i_d,-}\right).
\end{equation}
In view of \eqref{eq_t1tris} and \eqref{eq_LemmaN} we get
\begin{equation}
\label{eq_t1bis}
\Ha^{d-1}(\tilde J_n^-)\geq 1-3\eps.
\end{equation}
\par
Finally, if we set
$$
G_{n,\eps}:=\{x\in \tilde J^{-}_n\,:\, |\nu_{u_n}\cdot \xi_{\bar i}|> \eps\}\qquad\text{and}\qquad B_{n,\eps}:=\tilde J^-_n\setminus G_{n,\eps},
$$
coming back to \eqref{eq_area} we have
$$
\Ha^{d-1}(G_{n,\eps})+\eps^2\Ha^{d-1}(B_{n,\eps})>1-3\eps,
$$
so that
$$
\Ha^{d-1}(G_{n,\eps})>1-3\eps-\eps^2C,
$$
where $C:=\sup_n \Ha^{d-1}(J_{u_n})<+\infty$. Finally we orient the normal $\nu_{u_n}$ on $G_{n,\eps}$ in such a way that
$$
\nu_{u_n}\cdot \xi_{\bar i}>\eps.
$$
The inequalities \eqref{eq_orientation} then also hold true on $G_{n,\eps}$ if $\delta$ is small enough thanks to \eqref{eq_ed-delta}.
Reducing $\tilde J^-_n$ to $G_{n,\eps}$ if necessary, the full claim follows taking into account \eqref{eq_defJm} and \eqref{eq_t1bis}.

\vskip10pt\noindent{\bf Step 5.} Let $\tilde J^-_n\subseteq J_{u_n}$ be the set given by Step 4. Since the points of this set are seen from $H^-_\eps$ under $d$ different directions, in view of \eqref{eq_components} we infer that there exists $\tilde \eta_n\to 0$ such that for every $x\in \tilde J^-_n$
$$
|u^-_n(x)-u^-|<\tilde \eta_n.
$$
Reasoning in a similar way starting from the upper part $H^+_\eps$, and employing the opposite directions $\{-\xi_i\,:\, i=1,\dots,N\}$, we can construct $\tilde J^+_n\subseteq J_{u_n}$ with $\nu_{u_n}$ oriented such that again
$$
e_d\cdot \nu_{u_n}>0\qquad\text{and}\qquad \xi_i\cdot \nu_{u_n}>0 \text{ for every $i=1,\dots,N$},
$$
such that
\begin{equation}
\label{eq_t1+}
\Ha^{d-1}(\tilde J^+_n)\geq 1-c_\eps
\end{equation}
with $c_\eps \to 0$ as $\eps\to 0$, and such that for every $x\in \tilde J^+_n$
$$
|u^+_n(x)-u^+|<\tilde \eta_n.
$$
Notice that for $x\in \tilde J^-_n\cap \tilde J^+_n$, the orientation chosen is compatible with that of \eqref{eq_orientation}, so that indeed $u_n^-(x)$ and $u_n^+(x)$ are the two traces of $u_n$ at $x$.
\par
We can thus write, in view of the continuity of $\phi$
\begin{align*}
&\int_{J_{u_n}}[\phi(u_n^+)+\phi(u_n^-)]\,d\Ha^{d-1} \geq \int_{\tilde J_{n}^+\cap \tilde J_n^-}[\phi(u_n^+)+\phi(u_n^-)]\,d\Ha^{d-1} +
\int_{\tilde J_{n}^+\Delta \tilde J_n^-}[\phi(u_n^+)+\phi(u_n^-)]\,d\Ha^{d-1}\\
&\geq \int_{\tilde J_{n}^+\cap \tilde J_n^-}[\phi(u_n^+)+\phi(u_n^-)]\,d\Ha^{d-1}+\int_{\tilde J_{n}^+\setminus \tilde J_n^-}\phi(u_n^+)\,d\Ha^{d-1}+\int_{\tilde J_{n}^-\setminus \tilde J_n^+}\phi(u_n^-)\,d\Ha^{d-1}\\
&\geq \int_{\tilde J_{n}^+}\phi(u_n^+)\,d\Ha^{d-1}+\int_{\tilde J_{n}^-}\phi(u_n^-)\,d\Ha^{d-1}\\
&\geq [\phi(u^+)-\tilde \eta_n]\Ha^{d-1}(\tilde J_{n}^+)+[\phi(u^-)-\tilde \eta_n]\Ha^{d-1}(\tilde J_{n}^-)
\end{align*}
where $\tilde\eta_n\to 0$, so that, taking into account \eqref{eq_t1} and \eqref{eq_t1+}
$$
\liminf_{n\to+\infty} \int_{J_{u_n}}[\phi(u_n^+)+\phi(u_n^-)]\,d\Ha^{d-1}\geq [\phi(u^+)+\phi(u^-)](1-2c_\eps).
$$
The conclusion follows by letting $\eps\to 0$.
\end{proof}

 In the proof of Theorem \ref{thm:lsc-phi}  we made use of the following abstract lemma.

 \begin{lemma}
\label{lem:cov}
Let $(X,\mathcal{M},\mu)$ be a finite measure space. Let $\eps>0$ and $d\geq 2$. Then there exists $N\in \N$ that only depends on $\mu(X),\eps,d$ such that if $\{E_i\}_{i=1,\dots, N}$ is a family of sets in $\mathcal{M}$, we can find $\bar i$ such that
$$
\mu\left( E_{\bar i} \setminus \bigcup_{j_1<j_2<\dots<j_d} (E_{j_1}\cap E_{j_2}\cap\dots\cap E_{j_d})\right)<\eps.
$$
\end{lemma}

\begin{proof}
Up to dividing $\eps$ by $\mu(X)$ we suppose without loss of generality that $\mu(X)=1$. It is enough to prove that for any $d\geq 2,\eps>0$, there is some $N(d,\eps)\ge 1$ such that any family of $N\geq N(d,\eps)$ of sets $(E_i)_{1\leq i\leq N}$ there is some $i$ that verifies
\[
\mu\left(E_i\setminus\bigcup_{J\subset [1,N]\setminus \{i\},|J|=d-1}\bigcap_{j\in J}E_j\right)<\eps,
\]
meaning that there is some $i$ such that every point of $E_i$ outside a set of measure  less than  $\eps$ is in (at least) $d-1$ other sets $E_j$ (for $j\neq i$). 
\par
We prove it by recursion. If $d=2$, let $N:=\left[\frac{1}{\eps}\right]$,  where $[\cdot]$ denotes the integer part. Given  $(E_i)_{1\leq i\leq N}$, let us consider the sets $\left(E_{i}\setminus\bigcup_{1\leq j\leq N,j\neq i}E_{j}\right)_{1\leq i\leq N}$. These are disjoint and $\mu(X)=1$, so there is some $i$ such that
\[
\mu\left(E_{i}\setminus\bigcup_{1\leq j\leq N,j\neq i}E_{j}\right)\leq \frac{1}{N}\leq\eps,
\]
which proves the initialisation. 
\par
 Assume now that the result is true for $d$ and let us check it for $d+1$. Let

\[
N:=N\left(d,\frac{\eps}{2 }\right)\qquad\text{and}\qquad M:=\left[\frac{2 }{\eps}\right],
\] 
and let us consider $N\times M$ sets that we classify into $N$ groups of $M$ sets, written $(E_{k,i})_{1\leq k\leq N,1\leq i\leq M}$. For every $k\in [ 1,N]$, the sets $\left(E_{k,i}\setminus\bigcup_{1\leq j\leq M,j\neq i}E_{k,j}\right)_{1\leq i\leq M}$ are disjoints so there is some $i_k$ such that
\[\mu\left(E_{k,i_k}\setminus\bigcup_{1\leq i\leq M,i\neq i_k}E_{k,i}\right)\leq \frac{1}{M}\leq \frac{\eps}{2}.\]
Considering the sets $(E_{k,i_k})_{1\leq k\leq N}$, since $N=N\left(d,\frac{\eps}{2}\right)$ we find some $\overline{k}$ such that
\[
\mu\left(E_{\overline{k},i_{\overline{k}}}\setminus\bigcup_{K\subset [1,N]\setminus \{\overline{k}\},|K|=d-1}\bigcap_{k\in K}E_{k,i_k}\right)\leq\frac{\eps}{2}.
\]
This means that outside a set of measure at most $\frac{\eps}{2}$, every point of $E_{\overline{k},i_{\overline{k}}}$ is in $d-1$ sets of the form $E_{k,i_k}$ for $k\neq\overline{k}$, and similarly every point outside a set of measure at most $\frac{\eps}{2}$ is also in one set of the form $E_{\overline{k},i}$ for some $i\neq i_{\overline{k}}$. We conclude that outside of measure at most $\eps$, every point of $E_{\overline{k},i_{\overline{k}}}$ belongs to $d$ other sets, meaning $N(d+1,\eps)$ is well-defined and $N(d+1,\eps)\leq N\left(d,\frac{\eps}{2}\right)\left[\frac{2 }{\eps}\right]$.
\end{proof}
 
  \begin{remark}
\label{rem:blow}
{\rm

Let us detail the blow up argument used in the proof of Theorem \ref{thm:lsc-phi}. If we set
$$
\mu_n:=[\phi(u_n^+)+\phi(u_n^-)]\Ha^{d-1}\lfloor J_{u_n}
$$
and assume that (up to a subsequence)
$$
\mu_n \weakst \mu \qquad\text{weakly* in }\mathcal{M}_b(\Om)
$$ 
for some Radon measure $\mu$ on $\Om$, the conclusion follows if we show that
$$
\mu\ge [\phi(u^+)+\phi(u^-)]\Ha^{d-1}\lfloor J_{u}\qquad\text{as measures on }\Om.
$$
With this aim is sufficient to show that
\begin{equation}
\label{eq:RN}
\frac{d\mu}{d\Ha^{d-1}}(x)\ge [\phi(u^+(x))+\phi(u^-(x))]\qquad\text{for $\Ha^{d-1}$-a.e. $x\in J_u$},
\end{equation}
where $\frac{d\mu}{d\Ha^{d-1}}$ denotes the Radon-Nykodim derivative of $\mu$ with respect to $\Ha^{d-1}$ (restricted to $J_u$).
\par
Let us assume (up to subsequences) that
$$
\lambda_n:=\Ha^{d-1}\lfloor J_{u_n} \weakst \lambda  \qquad\text{weakly* in }\mathcal{M}_b(\Om),
$$
and that
$$
|e(u_n)|\,dx\weakst f\,dx\qquad\text{weakly* in }\mathcal{M}_b(\Om),
$$
where $f\in L^1(\Om)$ (this is possible since $(e(u_n))_{n\in\N}$ is bounded in $L^2$).
\par
Let $x\in J_u$ be such that
$$
\frac{d\mu}{d\Ha^{d-1}}(x)=\lim_{r\to 0}\frac{\mu(Q_{x,r})}{r^{d-1}},
\qquad
\lim_{r\to 0}\frac{\lambda(Q_{x,r})}{r^{d-1}}<+\infty,
\qquad
\lim_{r\to 0}\frac{1}{r^{d-1}}\int_{Q_r(x)}|f|\,dx=0,
$$
and (having choosen the axis so that $\nu_u(x)=e_d$), for $r\to 0^+$
$$
u(x+r \cdot) \to u^+(x)1_{Q_1^+}+u^-(x)1_{Q_1^-}\qquad\text{strongly in }L^1(Q_1;\R^d).
$$
Since $\Ha^{d-1}$-a.e. $x\in J_u$ satisfies these properties, it suffices to concentrate on such points to prove inequality \eqref{eq:RN}.
\par
Let $r_k\to 0$ be such that
$$
\mu(\partial Q_{x,r_k})=\lambda(\partial Q_{x,r_k})=0.
$$
Since by weak convergence and the relation above we have $\mu_n(Q_{x,r_k})\to \mu(Q_{x,r_k})$, and similarly for $\lambda$, we can choose $n_k\nearrow +\infty$ such that 
$$
\mu(Q_{x,r_k})\ge \mu_{n_k}(Q_{x,r_k})-\frac{r_k^{d-1}}{k},
\qquad
\lambda(Q_{x,r_k})\ge \lambda_{n_k}(Q_{x,r_k})-\frac{r_k^{d-1}}{k},
$$
and
$$
\int_{Q_{x,r_k}}|f|\,dx\ge \int_{Q_{x,r_k}}|e(u_{n_k}|\,dx-\frac{r_k^{d-1}}{k}.
$$
Moreover, setting $v_k(y):=u_{n_k}(x+r_k y)$ we can assume also
$$
v_k \to u^+(x)1_{Q_1^+}+u^-(x)1_{Q_1^-}\qquad\text{strongly in }L^1(Q_1;\R^d).
$$
We get
$$
\int_{Q_1}|e(v_k)|\,dx=\frac{1}{r_k^{d-1}} \int_{Q_{x,r_k}}|e(u_{n_k}|\,dx\le \frac{1}{r_k^{d-1}}\int_{Q_{x,r_k}}|f|\,dx+\frac{1}{k}\to 0
$$
and
$$
\Ha^{d-1}(J_{v_k})=\frac{1}{r_k^{d-1}}\Ha^{d-1}(J_{u_{n_k}}\cap Q_{x,r_k})=\frac{\lambda_{n_k}(Q_{x,r_k})}{r_k^{d-1}}\le \frac{\lambda(Q_{x,r_k})}{r_k^{d-1}}+\frac{1}{k}\to c<+\infty,
$$
so that, using the lower semicontinuity \eqref{eq:blow-sci} concerning functions on the unit square (and to which the proof of the Theorem has been reduced)
\begin{align*}
\frac{d\mu}{d\Ha^{d-1}}(x)&=\lim_{k\to+\infty}\frac{\mu(Q_{x,r_k})}{r_k^{d-1}}\ge \liminf_{k\to+\infty}\frac{\mu_{n_k}(Q_{x,r_k})}{r_k^{d-1}}\\
&=\liminf_{k\to+\infty} \int_{J_{v_k}}[\phi(v_k^+)+\phi(v_k^-)]\,d\Ha^{d-1}\ge
\phi(u^+(x))+\phi(u^-(x))
\end{align*}
and \eqref{eq:RN} follows.
 
}
\end{remark}
\section{Existence of minimizers: proof of Theorem \ref{main1}}
\label{sec_proof-main}

We are now in a position to prove the first main result of the paper.
\begin{proof}[Proof of Theorem \ref{main1}]
Let $(E_n,u_n)_{n\in \N}$ be a minimizing sequence: since the function $f$ is not identically equal to $+\infty$, and in view of Remark \ref{rem:nonempty}, there exists $C>0$ such that
$$
\J(E_n,u_n)\leq C.
$$
Since $u_n=0$ a.e. on $E_n$ we may write
$$
\int_{\partial^*E_n}|u_n^+|^2\,d\Ha^{d-1}+\int_{J_{u_n}\setminus \partial^*E_n}[|u_n^+|^2+|u_n^-|^2]\,d\Ha^{d-1}=
\int_{J_{u_n}}[|u_n^+|^2+|u_n^-|^2]\,d\Ha^{d-1}
$$
so that we infer
$$
\Ha^{d-1}(\partial^*E_n)\leq C\qquad\text{and}\qquad \int_\Om |e(u_n)|^2\,dx+\Ha^{d-1}(J_{u_n})+\int_{J_{u_n}}[|u_n^+|^2+|u_n^-|^2]\,d\Ha^{d-1}\leq C.
$$
Notice that
\begin{align*}
|E(u_n)|(\Om')&=\int_{\Om'}|e(u_n)|\,dx+\int_{J_{u_n}}|u_n^+-u_n^-|\,d\Ha^{d-1}\\
&\leq \int_{\Om'\setminus \Om}|e(V)|\,dx+
\int_\Om |e(u_n)|\,dx+\int_{J_{u_n}}[|u_n^+|+|u_n^-|]\,d\Ha^{d-1}\\
&\leq \int_{\Om'\setminus \Om}|e(V)|\,dx+\frac{1}{2}\left[|\Om|+\int_\Om |e(u_n)|^2\,dx+ 2\Ha^{d-1}(J_{u_n})+\int_{J_{u_n}}[|u_n^+|^2+|u_n^-|^2]\,d\Ha^{d-1}\right]
\leq \tilde C,
\end{align*}
for some $\tilde C>0$. Moreover, thanks to Theorem \ref{thm:embedding} applied to $u-V$ we may assume also that
\begin{equation}
\label{eq_bound-2d}
\|u_n\|_{L^{\frac{2d}{d-1}}(\Om')}\leq \tilde C.
\end{equation}
By the compactness result in $SBD$ (see Theorem \ref{thm:sbd}), there exist a subsequence $(u_{n_k})_{k\in\N}$ and $u\in SBD(\Om')$ with $u=V$ on $\Om'\setminus \Om$ and such that
\begin{equation}
\label{eq_s1}
u_{n_k}\to u\qquad\text{strongly in }L^1(\Om';\R^d),
\end{equation}
\begin{equation}
\label{eq_weu}
e(u_{n_k})\rightharpoonup e(u)\qquad\text{weakly in }L^2(\Om';M^{d\times d}_{sym}),
\end{equation}
and
\begin{equation*}
\label{eq_lsc-jump}
\Ha^{d-1}(J_{u})\leq \liminf_{k\to+\infty} \Ha^{d-1}(J_{u_{n_k}}).
\end{equation*}
Concerning the sets $E_{n_k}$, we may assume, up to a further subsequence if necessary, that there exists a set of fine perimeter $E\subseteq \Om$ such that
\begin{equation}
\label{eq_s-ob}
1_{E_{n_k}}\to 1_E\qquad\text{strongly in }L^1(\R^d)
\end{equation}
with
$$
\Ha^{d-1}(\partial^*E)\leq \liminf_{k\to+\infty} \Ha^{d-1}(\partial^*E_{n_k}).
$$
In particular we get
\begin{equation}
\label{eq_lsc-f}
f(|E|)\leq \liminf_{n\to+\infty} f(|E_n|).
\end{equation}
Let us prove that 
\begin{equation}
\label{eq_u-adm}
 (E,u)\in \as(V) .
\end{equation}
In view of \eqref{eq_bound-2d} we infer that $u\in L^{\frac{2d}{d-1}}(\Om';\R^d)$ so that in particular $u\in L^2(\Om';\R^d)$. Moreover $u=V$ on $\Om'\setminus \Om$, while $u=0$ a.e. on $E$ thanks to \eqref{eq_s1} and \eqref{eq_s-ob}.
\par
Since the divergence constraint is intended in the sense of distributions on $\Om$, this passes easily to the limit thanks to \eqref{eq_s1}. Moreover, in view of Theorem \ref{thm:closure} we deduce
$$
u^\pm \perp \nu_u \qquad\text{on }J_u.
$$
In particular this entails
$$
u^+ \perp \nu_E \qquad\text{on }\partial^*E\cap \Om,
$$
since for $x\in \partial^*E$ we have either $x\in J_u$ or $u^+(x)=0$. We conclude that the non-penetration  constraint for the velocity field holds on $\partial^*E$ and on $J_u\setminus \partial^*E$, so that \eqref{eq_u-adm} holds true.
\par
Let us prove the pair $(E,u)$ is a minimizer for the problem. Thanks to \eqref{eq_weu} we get
$$
\int_{\Om'} |e(u)|^2\,dx\leq \liminf_{k\to+\infty} \int_{\Om'} |e(u_{n_k})|^2\,dx,
$$
while in view of  Theorem \ref{thm:lsc-phi} we have that
$$
\int_{J_u}[|u^+|^2+|u^-|^2]\,d\Ha^{d-1}\leq \liminf_{k\to+\infty}
\int_{J_{u_{n_k}}}[|u_{n_k}^+|^2+|u_{n_k}^-|^2]\,d\Ha^{d-1},
$$
which entails 
\begin{equation}
\label{eq_rob-sci}
\begin{split}
&\int_{\partial^*E}|u^+|^2\,d\Ha^{d-1}
+\int_{J_u\setminus \partial^*E}[|u^+|^2+|u^-|^2]\,d\Ha^{d-1}\\
&\leq \liminf_{k\to+\infty}\left[ \int_{\partial^* E_{n_k}}[u_{n_k}^+|^2\,d\Ha^{d-1}+\int_{J_{u_{n_k}}\setminus \partial^* E_{n_k}}[|u_{n_k}^+|^2+|u_{n_k}^-|^2]\,d\Ha^{d-1}\right]
\end{split}
\end{equation}
since $u=0$ a.e. on $E$ and $u_{n_k}=0$ a.e. on $E_{n_k}$.
\par
Let us prove that
\begin{equation}
\label{eq_perturbe}
2\Ha^{d-1}(J_u\setminus\partial^*E)+\Ha^{d-1}(\partial^*E)\leq\liminf_{k\to+\infty}\left(2\Ha^{d-1}(J_{u_{n_k}}\setminus\partial^*E_{n_k})+\Ha^{d-1}(\partial^*E_{n_k})\right).
\end{equation}
 Let us choose $h\in \R^d$ such that
\begin{multline*}
\Ha^{d-1}(\{x\in \partial^*E\cup J_u\,:\, u^+(x)=h\})=
\Ha^{d-1}(\{x\in \partial^*E\cup J_u\,:\, u^-(x)=h\})\\
=\Ha^{d-1}(\{x\in \partial^*E_{n_k}\cup J_{u_{n_k}}\,:\, u^+_{n_k}(x)=h\})=
\Ha^{d-1}(\{x\in \partial^*E_{n_k}\cup J_{u_{n_k}}\,:\, u^-_{n_k}(x)=h\})=0.
\end{multline*}
This is possible because for example the sets $\{x\in \partial^*E\cup J_u\,:\, u^+(x)=h\}$ are disjoint as $h$ varies, and similarly for the other sets. In particular, setting
\[
v^h:=u+h1_E\qquad\text{and}\qquad v^h_{n_k}:=u_{n_k}+h1_{E_{n_k}}
\]

we have
\begin{equation*}
\label{eq_claim-jump}
J_{v^h}=J_u\cup J_{1_E}=\partial^*E\cup J_u\qquad\text{and}\qquad J_{v^h_{n_k}}=J_{u_{n_k}}\cup J_{1_{E_{n_k}}}=\partial^*{E_{n_k}}\cup J_{u_{n_k}}
\end{equation*}
up to $\Ha^{d-1}$-negligible sets. 
  If we apply Theorem \ref{thm:lsc-phi} with the choice $\phi_h(s)=1_{\{s\neq h\}}$ to the sequence $(v^h_{n_k})_{k\in\N}$ 
we get
\begin{equation}
\label{eq:lsc2jump}
\begin{split}
\Ha^{d-1}(\partial^*E)+2\Ha^{d-1}(J_u\setminus \partial^*E)&=\int_{J_{v^h}}[\phi_h((v^h)^+)+\phi_h((v^h)^-)]d\Ha^{d-1}\\
&\leq\liminf_{k\to+\infty} \int_{J_{v^h_{n_k}}}[\phi_h((v^h_{n_k})^+)+\phi_h((v^h_{n_k})^-)]d\Ha^{d-1}\\
&=\liminf_{k\to+\infty}\left[ \Ha^{d-1}(\partial^*E_{n_k})+2\Ha^{d-1}(J_{u_{n_k}}\setminus \partial^* E_{n_k})\right]
\end{split}
\end{equation}
so that \eqref{eq_perturbe} holds true.

\par
Gathering \eqref{eq_weu}, \eqref{eq_rob-sci}, \eqref{eq_lsc-f} and \eqref{eq_perturbe}, we deduce
$$
\J(E,u)\leq \liminf_{k\to+\infty} \J(E_{n_k},u_{n_k})
$$
so that, taking into account \eqref{eq_u-adm}, the pair $(E,u)$ is a minimizer of the main problem \eqref{mainpb}, and the proof is concluded.
\end{proof}

\section{Regularity  of  two-dimensional minimizers: proof of Theorem \ref{main2}}\label{sec-reg}

This section is devoted to the proof Theorem \ref{main2} concerning the regularity of minimizers in dimension two.

As mentioned in the Introduction, the general strategy used by De Giorgi, Carriero and Leaci for the Mumford-Shah problem in \cite{DGCL} faces the new difficulties given by the vectorial context,  considered in \cite{CFI_19,CCI_19} in connection to the Griffith fracture problem, and also
by extra conditions  proper to our problem, that is incompressibility and non-penetration for the velocity fields. 
We follow the main lines of \cite{CFI_19,CCI_19}: however technical difficulties allow us to deal only with dimension $2$  (see point (a) below).
\par

Since our drag problem involves pairs $(E,u)$ as admissible configurations, and some points of $\partial^*E$ may not be jump points of $u$, it will be useful to deal with pairs $(J,u)$, where $J$ is a rectifiable set and $u$ is a function whose jumps are contained (up to $\Ha^1$-negligible sets) in $J$ and satisfy the constraints of zero divergence and non-penetration. More precisely we formulate the following definition.

\begin{definition}[\bf The class ${\mathcal V}$]
\label{def:nubar}
Let $\Om \subseteq \R^2$ be an open set. We say that $(J,u)\in \mathcal{V}(\Om)$ if $J\subseteq \Om$ is a rectifiable set, and $u\in SBD(\Om)$ is such that ${\rm div}\,u=0$ in the sense of distributions in $\Om$, $\Ha^1(J_u\setminus J)=0$ and $u^\pm_{|J}\cdot \nu_J=0$ $\Ha^1$-a.e. on $J$.
\end{definition}

\noindent
The structure of the section is the following. 
\begin{itemize}
\item[(a)] In Section \ref{sec:smooth} we prove a fundamental approximation lemma (Smoothing Lemma \ref{lem:smoothing1}), which allows us to approximate every $(J,u)\in {\mathcal V}(Q_1)$ with $\Ha^1(J)$ small by a configuration $(J\setminus Q_r,v)\in {\mathcal V}(Q_1)$, where $v$ is a Sobolev function in the slightly smaller square $Q_r$ with a control on the energy. The idea is that the jumps of $u$ in $Q_r$ are ``smoothed out'', giving rise to the function $v$ which preserves the divergence free constraint together with the non-penetration condition. This result is inspired by \cite{CCI_19}, and it is here that the dimension two is fundamental.
\item[(b)] In Section \ref{sec:griffith} we prove regularity for local minimizers of a Griffith functional 
$$
G(J,u):=\int_{\Om}|e(u)|^2dx+\Ha^1(J),
$$
defined on pairs $(J,u) \in {\mathcal V}(\Om)$. The kind of local minimality considered is very weak, and inspired by the kind of competitors
that can be constructed thanks to the Smoothing Lemma \ref{lem:smoothing1}. The key result to get regularity is given by the decay estimate contained in Proposition \ref{prop:decay}. 
\par\noindent
Regularity for minimizers of the Griffith energy is then used in Section \ref{sec:proof2} to prove Theorem \ref{main2}, that is to show the regularity of minimizers of the drag problem.
\item[(c)] Finally, motivated by the regularity result of Theorem \ref{main2}, in Section \ref{sec:strong} we describe a different relaxation of the drag problem which involves topologically closed obstacles and Sobolev velocities: the regularity result can be used to prove that such a formulation is well posed in dimension two.
\end{itemize}

\subsection{The smoothing lemma}
\label{sec:smooth}
We fix a standard radial,  smooth, nonnegative mollifier $\rho$ with support in a disc of radius $1/8$ and denote 
$$
\rho_\delta(x):=\delta^{-2}\rho\left(\frac x\delta\right).
$$

The main result of the section is the following smoothing lemma which is in the spirit of  \cite{CCI_19}. 

 \begin{lemma}[\bf Smoothing Lemma]
  \label{lem:smoothing1}
 There exist $C,\eta>0$ such that for any $(J,u) \in \mathcal{V}(Q_1)$ with $\Ha^1(J)<\eta$, then letting $\delta:=\Ha^1(J)^{\frac{1}{2}}$ 
 there exist $r\in ]1-\delta^{\frac{1}{2}},1[$ and $v\in SBD(Q_1) \cap H^1(Q_r)$ such that the following items hold true.
\begin{itemize}
\item[(a)] $\Ha^0(J\cap \partial Q_r)=0$ and for every $0<s<r$
$$
\Ha^1(J\cap (Q_r\setminus Q_{r-s}))\le C \delta^{\frac{3}{2}} s.
$$
\item[(b)] $\{v\neq u\}\subseteq Q_r$ and $(J\setminus Q_r, v)\in \mathcal{V}(Q_1)$.
\item[(c)] It holds
$$
\|e(v)\|_{L^2(Q_1)}\le (1+C\delta^\frac{1}{6})\|e(u)\|_{L^2(Q_1)}.
$$

\item[(d)] There exists a cut-off function $\varphi\in C^\infty(Q_r,[0,1])$ with
$\varphi=0$ on $Q_r\setminus Q_{r- \delta}$, $\varphi=1$ on $Q_{r- 4\delta}$, and such that
$$
\|e(v)- \varphi \rho_{\delta}*e(u)\|_{L^2(Q_r)} \le C\delta^{\frac{1}{6}} \|e(u)\|_{L^2(Q_1)}.
$$
\end{itemize}

\end{lemma}

\begin{proof}

The proof follows the strategy introduced in \cite{CCI_19}, and some parts will be referred directly to that paper. However, since our conclusion is slightly different, we prefer to develop some computations in detail. We will use the notation $a\lesssim b$ when $a\le Cb$ for some dimensional constant $C$.
\par
We divide the proof in several steps.

\vskip10pt\noindent
{\bf Step 1: Subdivision in small squares.}
Let us set
$$
N:=1+\left[\Ha^1(J)^{-\frac{1}{2}}\right],
$$
where $[\cdot]$ denotes the integer part. In the following we will assume that $\Ha^1(J)$ is arbitrary small, so that $N$ is arbitrarily large. For convenience in the construction, we will set $\delta=1/N { \le \Ha^1(J)^{\frac{1}{2}}}$, which (mildly) differs from the choice of the statement: { yet} since $\delta$ is asymptotically equivalent to $\Ha^1(J)^{\frac{1}{2}}$, the mismatch does not affect the validity of the conclusion.
\par
For $r\in ]1-\delta^{\frac{1}{2}}, 1[$ { and each $k\ge -2$}, let us set 
$$
\delta_k:=\frac{\delta r}{2^k}\qquad\text{and}\qquad r_k=\left(N-\frac{1}{2^k}\right)\delta \,.  
$$ 
Then we consider a partition { (up to a negligible set)} of $Q_r$ into cubes obtained by filling $Q_{r_0}$ with cubes of side $\delta_0$ and denoted by $(\tilde{q}_{0,j})_j$, and then each $Q_{r_k}\setminus Q_{r_{k-1}}$ with cubes of side $\delta_k$ and denoted $(\tilde{q}_{k,j})_j$ (note that there is only one way to do this).

\par
For any square $q=z+[-t,t]^2$, we write 
$$
q':=z+\left[-\frac{8}{7}t,\frac{8}{7}t\right]^2\qquad\text{and}\qquad q'':=(q')'.
$$ 
We will set 
$$
q_{k,j}:=(\tilde{q}_{k,j})'.
$$
We may notice that with our choices
\begin{equation}
\label{eq:squaresk}
\forall k\ge 1\,:\, q_{k,j}''\Subset Q_{r_{k+1}}\setminus Q_{r_{ k-2}},
\end{equation}
and $\{q_{k,j}''\}_{k,j}$ is a covering of $Q_r$ with a fixed finite number of overlapping: indeed each $q_{k,j}''$ meets at most $8$ neighbours $q_{p,i}''$, and they all verify $|k-p|\leq 1$, meaning $\delta_k/\delta_p\in \left\{\frac{1}{2},1,2\right\}$. This is because the factor $\frac{8}{7}$ above is chosen such that $\left(\frac{8}{7}\right)^3<\frac{3}{2}$.

\vskip10pt\noindent{\bf Step 2: Choice of the square $Q_r$.}
We now make a convenient choice of $r$ such that the density of $J$ near $\partial Q_r$ is small, following an approach similar to \cite[Theorem 2.1]{CFI_17}.
\par
We claim that there exist $C,\eta>0$ such that for $\delta<\eta$ we can choose $r\in ]1-\sqrt{\delta},1[$ with $\Ha^0(J\cap \partial Q_r)=0$, 
\begin{equation}
\label{eq:r0}
\forall\,s\in ]0,r[\,:\,\Ha^1(J\cap (Q_r\setminus Q_{r-s}))\le C \delta^{\frac{3}{2}} s
\end{equation}
and
\begin{equation}
\label{eq:r2}
\int_{Q_r\setminus Q_{r_{-2}}}|e(u)|^2\,dx< C\delta^{\frac{1}{2}} \int_{Q_1}|e(u)|^2\,dx.
\end{equation}

Consider indeed the measure $\mu$ on $[0,1]$ defined as
\[
\mu(E):=\frac{\Ha^1(J\cap Q_E)}{\Ha^1(J)}+\frac{\int_{Q_E}|e(u)|^2\,dx}{\int_{Q_1}|e(u)|^2\,dx},
\]
where $Q_E:=\cup_{r\in  E}\partial Q_r$ is the cubic shell associated to $E\subset [0,1]$. 
It suffices to prove that we can find $r\in ]1-\delta^{\frac{1}{2}}, 1[$ such that 
\begin{equation}
\label{eq:hszero}
\Ha^0(J\cap \partial Q_r)=0, 
\end{equation}
and, denoting  $I_r^s:=[r-s,r[$ for $0<s<r$,  
\begin{equation}
\label{eq:vitali}
 \mu(I_r^s)\le \hat C\delta^{-\frac{1}{2}} s,
\end{equation}
where $\hat C>0$ is a suitable constant which we fix below.
 Indeed, if $\delta$ is small enough this implies that (recall that $\Ha^1(J)$ behaves like $\delta^2$)
$$
\Ha^1(J\cap(Q_r\setminus Q_{r-s}))\le  \Ha^1(J) \mu(I_r^s) \le \hat C\delta^{\frac{3}{2}} s
$$
and
$$
\int_{Q_r\setminus Q_{r-4\delta r}}|e(u)|^2\,dx\leq \hat C\delta^{-\frac{1}{2}} (4\delta r)\int_{Q_1}|e(u)|^2\,dx\le 4\hat C\delta^{\frac{1}{2}} \int_{Q_1}|e(u)|^2\,dx,
$$
so that \eqref{eq:r0} and \eqref{eq:r2} follow by choosing $C:=4\hat C$.

\par
Let $I_1$ be the union of all intervals that do not satisfy \eqref{eq:vitali}. If $(I_{r_i}^{s_i})$ is a {\it Vitali covering} of $I$, then
\[
2=\mu([0,1])\geq \sum_{i}\mu(I_{r_i}^{s_i})>\hat C \delta^{-\frac{1}{2}}\sum_i |I_{r_i}^{s_i}|=
 \frac{\hat C\delta^{-\frac{1}{2}}}{5}\sum_i |5I_{r_i}^{s_i}| \geq \frac{\hat C\delta^{-\frac{1}{2}}}{5}|I_1|\,
\]
hence $|I_1|< \frac{10}{\hat C}\delta^{\frac{1}{2}}$.
\par
 Let $I_2:=\pi_x(J) \cup \pi_y(J)$, where $\pi_x,\pi_y$ denote the projection on the coordinate axis: we have asymptotically $|I_2|\le 2\delta^2$. If $C>10$, this implies that for $\delta$ small enough
$$
]1-\sqrt \delta,1[\setminus (I_1\cup I_2)\not=\emptyset,
$$
which yields the existence of $r$ which verifies claims \eqref{eq:hszero} and \eqref{eq:vitali}.

\vskip10pt\noindent{\bf Step 3: A first approximation.}
In view of \eqref{eq:r0} and of \eqref{eq:squaresk}, for every $k\geq 1$ we have
$$
\Ha^1(J_u\cap q_{k,j}'')\lesssim \delta^{\frac{3}{2}}\delta_k, 
$$
while if $\delta$ is small enough (recall that $\Ha^1(J)$ behaves like $\delta^2$ and $r\in]1-\delta^{\frac{1}{2}},1[$)
$$
\Ha^1(J_u\cap q_{0,j}'')\le \Ha^1(J_u)\lesssim \delta \delta_0. 
$$
This means that the jump set of $u$ in every cube of the constructed subdivision is arbitrarily small compared to its sides.

\par
Thanks to  
{ \cite[Proposition~3]{CCF16},}
 and taking into account the preceding inequalities , for every $(k,j)$ there is a set $\om_{k,j}\subset q_{k,j}'$ and an affine function $a_{k,j}$ with $e(a_{k,j})=0$, such that
\begin{equation}
\label{eq:omk}
|\om_{k,j}|\lesssim\delta_k\Ha^1(J_u\cap q_{k,j}'') \lesssim \delta\delta_k^2
\end{equation}
\begin{equation}
\label{eq:L4}
\int_{q_{k,j}'\setminus \om_{k,j}}|u-a_{k,j}|^{4}\,dx\lesssim\left(\delta_k\int_{q_{k,j}''}|e(u)|^2\,dx\right)^2,
\end{equation}
and the function $v_{k,j}:=u+(a_{k,j}-u)1_{\om_{k,j}}$ verifies
\begin{equation}
\label{eq:wkj}
\begin{split}
\int_{q_{k,j}}|e(\rho_{\delta_k}*v_{k,j})-\rho_{\delta_k}*e(u)|^2\,dx&\lesssim\left(\frac{\Ha^1(J_u\cap q_{k,j}'')}{\delta_k}\right)^\frac{1}{3}\int_{q_{k,j}''}|e(u)|^2\,dx\\
&\lesssim\delta^{\frac{1}{3}}\int_{q_{j,k}''}|e(u)|^2\,dx,
\end{split}
\end{equation}
{ (see~\cite[p.~1389]{CCF16})}
where $\rho$ is the mollifier defined  at the beginning of the section.
\par
Notice that in view of our construction (namely the choice of $r$), we have
\begin{equation}
\label{eq:est}
|\om_{k,j}|\ll |q_{k,j}|,
\end{equation}
and this is where we most use the fact that we are in two dimensions.
\par
We now let $(\varphi_{k,j})$ be a partition of unity associated to the covering $(q_{k,j})$ of $Q_r$ and such that $|\nabla \varphi_{k,j}|\lesssim \frac{1}{\delta_k}$. Let us set
\[
w:=1_{Q_1\setminus Q_r}u+1_{Q_r}\sum_{k,j}\varphi_{k,j}w_{k,j}\qquad\text{where}\qquad w_{k,j}:=\rho_{\delta_k}*v_{k,j}.
\]
 
\par
 
We claim that

\begin{equation}
\label{eq:claimw1}
w\in SBD(Q_1)\cap H^1(Q_r),\qquad \{w\not= u\}\subset Q_r,\qquad  \Ha^1(J_w\setminus J)=0,
\end{equation}

\begin{equation}
\label{eq:claimw2}
\left\|e(w)-\sum_{k,j}\varphi_{k,j}\rho_{\delta_k}* e(u)\right\|_{L^2(Q_r)}\lesssim\delta^\frac{1}{6} \|e(u)\|_{L^2(Q_1)},
\end{equation}
 and
\begin{equation}
\label{eq:claimw3}
\text{the trace of $w$ and $u$ on $\partial Q_r$ coincide.}
\end{equation}
We postpone the proof of these claims to Step 5.
\par
 
Let us set
\[
\varphi:=\sum_{(0,j)\in {\mathcal K}}\varphi_{0,j},
\]
where ${\mathcal K}$ denotes the set of indices such that $q_{0,j}$ has a distance greater than $2\delta r$ from $\partial Q_{r}$.
Since $r\in ]1-\delta^{\frac{1}{2}},1[$, in view of the definition of the set of indices $\mathcal K$, we get that the function $\varphi$ vanishes on $Q\setminus Q_{r-\delta}$ and it is equal to $1$ on $Q_{r-4\delta}$.
\par
We can write
\[
e(w)-\sum_{k,j}\varphi_{k,j}\rho_{\delta_k}* e(u)=\Big[e(w)-\varphi \rho_\delta * e(u)\Big]-\sum_{(k,j)\not\in {\mathcal K}}\varphi_{k,j}\rho_{\delta_k}* e(u).
\]
Thanks to \eqref{eq:r2} we have

\begin{align*}
\left\Vert\sum_{(k,j)\not\in {\mathcal K}}\varphi_{k,j}\rho_{\delta_k}* e(u)\right\Vert_{L^2(Q_r)}^2&=\left\Vert\sum_{(k,j)\not\in {\mathcal K}}\varphi_{k,j}\rho_{\delta_k}* e(u)\right\Vert_{L^2(Q_r\setminus Q_{r-2\delta r})}^2\\
&\lesssim\sum_{(k,j)\not\in {\mathcal K}}\left\|\varphi_{j,k}\rho_{\delta_k}*e(u)\right\|^2_{L^2(Q_1\setminus Q_{r-2\delta r})}\\
&\lesssim\|e(u)\|^2_{L^2(Q_1\setminus Q_{r-3\delta r})}\lesssim \delta^{\frac{1}{2}}\|e(u)\|^2_{L^2(Q_1)},
\end{align*}
so that in view of \eqref{eq:claimw2} we conclude
\begin{equation}
\label{eq:est-convw}
\Vert e(w)-\varphi \rho_\delta * e(u)\Vert_{L^2(Q_r)}\lesssim\delta^{\frac{1}{6}} \Vert e(u)\Vert_{L^2(Q_1)}.
\end{equation}
Moreover we may write
\begin{align*}
\|e(w)\|_{L^2(Q_r)}&\le \Vert \varphi \rho_\delta * e(u)\Vert_{L^2(Q_r)}+\Vert e(w)-\varphi \rho_\delta * e(u)\Vert_{L^2(Q_r)}\\
&=\Vert \varphi \rho_\delta * e(u)\Vert_{L^2(Q_{r-\delta})}+\Vert e(w)-\varphi \rho_\delta * e(u)\Vert_{L^2(Q_r)}\\
&\le \|e(u)\|_{L^2(Q_{1})}+\Vert e(w)-\varphi \rho_\delta * e(u)\Vert_{L^2(Q_r)},
\end{align*}
so that taking into account \eqref{eq:est-convw} we deduce
\begin{equation}
\label{eq:wpointc}
\|e(w)\|_{L^2(Q_1)}\le (1+C\delta^{\frac{1}{6}}) \Vert e(u)\Vert_{L^2(Q_1)},
\end{equation}
where $C>0$.

\vskip10pt\noindent
{\bf Step 4: Enforcing the divergence free constraint.} By admissibility, $u$ is divergence free in the sense of distributions in $Q_1$, so that the trace of $e(u)$ is zero in $Q_1$, while
\begin{equation}
\label{eq:trace-zero}
\int_{\partial Q_r}u\cdot \nu\,d\Ha^1=0,
\end{equation}
where $\nu$ is the outward normal vector of $Q_r$, and $u$ denotes the trace on $\partial Q_r$ ($J$ does not intersect $\partial Q_r$ by construction).
\par
Recalling that $w\in H^1(Q_r)$, we may write thanks to \eqref{eq:est-convw}
$$
\|\text{div}\,w\|_{L^2(Q_r)}=\|\text{Tr}(e(w))\|_{L^2(Q_r)}=\|\text{Tr}(e(w)-\varphi \rho_\delta*e(u))\|_{L^2(Q_r)}\lesssim  \delta^{\frac{1}{6}} \Vert e(u)\Vert_{L^2(Q_1)}.
$$
By \eqref{eq:claimw3} the trace of $u$ on $\partial Q_r$ coincides with that of $w$, so that from \eqref{eq:trace-zero} we deduce
\[
\int_{Q_r} \text{div}\, w \,dx=0.
\]
Using { a classical result (recorded at the end of this proof in Lemma~\ref{lem:necas})}, there exists a vector field $q\in H^1_0(Q_r)$ such that
\begin{equation}
\label{eq:bound-q}
\text{div}\,q=\text{div}\,w \qquad\text{and}\qquad  \Vert \nabla q\Vert_{L^2(Q_r)}\lesssim\Vert\text{div}\,w\Vert_{L^2(Q_r)}\lesssim \delta^\frac{1}{6} \|e(u)\|_{L^2(Q_1)}.
\end{equation}

\par
Let 
$$
v:=
\begin{cases}
w-q&\text{in }Q_r\\
u &\text{in }Q_1\setminus Q_r,
\end{cases}
$$
and let us check that $v$ satisfies the conclusions of the lemma. 
\par
The choice of $r$ given by Step 2 yields immediately point (a). Clearly $v\in SBD(Q_1)\cap H^1(Q_r)$ with $\{v\not=u\}\subseteq Q_r$. Moreover, since the trace of $w-q$ and $u$ coincide on $\partial Q_r$, we get  $div\,v=0$ in the sense of distributions in $Q_1$, so that point (b) is proved.
Points (c) and (d) follow from the corresponding properties for $w$ (see \eqref{eq:est-convw} and \eqref{eq:wpointc}) taking into account that the correction term $q$ has a small gradient norm of the order $\delta^{\frac{1}{6}}$ as estimated in \eqref{eq:bound-q}.

\vskip10pt\noindent
{\bf Step 5: Proof of the claims \eqref{eq:claimw1}, \eqref{eq:claimw2} and \eqref{eq:claimw3}.} In order to conclude the proof, we need to check the claims on the function $w$ contained in Step 3.
\par
Let us start by noticing that the oscillation of the maps $a_{k,j}$ on intersecting squares can be estimated. Indeed as soon as $q_{k,j}$ and $q_{p,i}$ intersects, then 
$$
|q_{k,j}\cap q_{p,i}|\gtrsim \max(|q_{k,j}|,|q_{p,i}|),
$$ 
and since  (see \eqref{eq:est}) 
$$
| (q'_{k,j}\cap q'_{p,i}) \cap (\om_{k,j}\cup\om_{p,i})|\ll | q'_{k,j}\cap q'_{p,i}|
$$ 
and $a_{j,k},a_{i,p}$ are affine, then using \cite[Lemma 3.4]{CCI_19}  and \eqref{eq:L4} we deduce
\begin{multline}
\label{eq:aL4}
\Vert a_{k,j}-a_{p,i}\Vert_{L^4( q'_{k,j}\cap q'_{p,i})}\lesssim
\Vert a_{k,j}-a_{p,i}\Vert_{L^4(( q'_{k,j}\cap q'_{p,i})\setminus (\om_{k,j}\cup \om_{p,i}))}\\
\leq
\Vert a_{k,j}-u\Vert_{L^4( q'_{k,j}\setminus \om_{k,j})}+\Vert a_{p,i}-u\Vert_{L^4( q'_{p,i}\setminus\om_{p,i})}
\lesssim\delta_k^{\frac{1}{2}}\Vert e(u)\Vert_{L^2(q_{k,j}'')}+\delta_p^{\frac{1}{2}}\Vert e(u)\Vert_{L^2(q_{p,i}'')}\\
\lesssim\delta_k^{\frac{1}{2}}\Vert e(u)\Vert_{L^2(q_{k,j}''\cup q_{p,i}'')},
\end{multline}
 as $\delta_k$ and $\delta_p$ are comparable.
 
\par
Let us come to the claims. Clearly
\begin{equation*}
\label{eq_ew1}
e(w)=\sum_{k,j}\varphi_{k,j}e(w_{k,j})+\sum_{k,j}\nabla \varphi_{k,j} \odot w_{k,j},
\end{equation*}
so that
\begin{equation}
\label{eq_ew}
e(w)-\sum_{k,j}\varphi_{k,j}\rho_{\delta_k}* e(u)\\
=
\sum_{k,j}\varphi_{k,j}\Big[e(w_{k,j})-\rho_{\delta_k}*e(u)\Big]+\sum_{k,j}\nabla \varphi_{k,j} \odot w_{k,j}.
\end{equation}
For the first term of the right hand side, we have thanks to \eqref{eq:wkj}
\begin{equation}
\label{eq:est-1term}
\begin{split}
&\left\|\sum_{k,j}\varphi_{k,j}\Big[e(w_{k,j})-\rho_{\delta_k}*e(u)\Big]\right\|^2_{L^2(Q_r)}\lesssim \sum_{k,j}\left\|\varphi_{k,j}\Big[e(w_{k,j})-\rho_{\delta_k}*e(u)\Big]\right\|^2_{L^2(Q_r)}\\
&\le \sum_{k,j}\left\|e(w_{k,j})-\rho_{\delta_k}*e(u)\right\|^2_{L^2(q_{k,j})}\le \delta^{\frac{1}{3}}\sum_{k,j} \|e(u)\|^2_{L^2(q''_{k,j})}\lesssim \delta^{\frac{1}{3}}\|e(u)\|^2_{L^2(Q_r)},
\end{split}
\end{equation}
where we used the finite overlapping of the squares $q''_{k,j}$ { for the first and
last estimates}.
\par
Let us estimate the second term on the right hand side of \eqref{eq_ew}. Notice that we may write
\begin{equation*}
\label{eq:2term12}
\sum_{k,j}\nabla \varphi_{k,j} \odot w_{k,j}=\sum_{q_{k,j}\cap q_{p,i}\neq\emptyset}\nabla \varphi_{k,j} \odot (w_{k,j}-w_{p,i})
\qquad\text{on $q_{p,i}$}
\end{equation*}
since $\sum_{k,j}\nabla \varphi_{k,j}=0$.

\begin{itemize}
\item[(a1)] If $q_{p,i}''\Subset Q_{r_{-1}}$, then $q_{j,k}\cap q_{i,p}\neq \emptyset$ means that $\delta_k=\delta_p=\delta$, $k=p=0$, and we may rewrite
 the term as
$$
\sum_{q_{0,j}\cap q_{0,i}\neq\emptyset}\nabla \varphi_{0,j} \odot (w_{0,j}-w_{0,i})
$$
We get
\begin{equation}
\label{eq:2term1}
\left\|\sum_{q_{0,j}\cap q_{0,i}\neq\emptyset}\nabla \varphi_{0,j} \odot (w_{0,j}-w_{0,i})\right\|^2_{L^2(q_{0,i})}
\lesssim \sum_{q_{0,j}\cap q_{0,i}\neq\emptyset} \frac{1}{\delta^2}\|w_{0,j}-w_{0,i}\|^2_{L^2(q_{0,j}\cap q_{0,i})}.
\end{equation}
Now
$$
\|w_{0,j}-w_{0,i}\|_{L^2(q_{o,j}\cap q_{0,i})}=\|\rho_{\delta}*(v_{0,j}-v_{0,i})\|_{L^2(q_{0,j}\cap q_{0,i})}\le \|v_{0,j}-v_{0,i}\|_{L^2(q'_{0,j}\cap q'_{0,i})}.
$$
Since
\begin{align*}
 \|v_{0,j}-v_{0,i}\|_{L^2(q'_{0,j}\cap q'_{0,i})}& \le \|(a_{0,j}-a_{0,i})1_{\omega_{0,j}\cup \omega_{0,i}}\|_{L^2(q'_{0,j}\cap q'_{0,i})} +\|(u-a_{0,j})1_{\omega_{0,i}}\|_{L^2(q'_{0,j}\setminus \omega_{0,j})}
 \\
 &+\|(u-a_{0,i})1_{\omega_{0,j}}\|_{L^2(q'_{0,i}\setminus \omega_{0,i})} \\
&\le \|(a_{0,j}-a_{0,i})\|_{L^4(q'_{0,j}\cap q'_{0,i})}|\omega_{0,j}\cup \omega_{0,i}|^{\frac{1}{4}} 
+\|(u-a_{0,j})\|_{L^4(q'_{0,j}\setminus \omega_{0,j})}|\omega_{0,i}|^{\frac{1}{4}}
\\
&+\|(u-a_{0,i})\|_{L^4(q'_{0,i}\setminus \omega_{0,i})}|\omega_{0,j}|^{\frac{1}{4}},
\end{align*}
recalling \eqref{eq:omk}, \eqref{eq:L4} and \eqref{eq:aL4} we get
$$
\|w_{0,j}-w_{0,i}\|_{L^2(q_{o,j}\cap q_{0,i})}\le \|v_{0,j}-v_{0,i}\|_{L^2(q'_{0,j}\cap q'_{0,i})}\le \delta^{1+\frac{1}{4}}\|e(u)\|_{L^2(q''_{0,j}\cup q''_{0,i})}.
$$
Coming back to \eqref{eq:2term1} we infer
\begin{equation}
\label{eq:estq0i}
\begin{split}
\left\|\sum_{k,j}\nabla \varphi_{k,j} \odot w_{k,j}\right\|^2_{L^2(q_{0,i})}&\le
\left\|\sum_{q_{0,j}\cap q_{0,i}\neq\emptyset}\nabla \varphi_{0,j} \odot (w_{0,j}-w_{0,i})\right\|^2_{L^2(q_{0,i})}\\
&\lesssim \delta^{\frac{1}{2}}\sum_{q_{0,j}\cap q_{0,i}\neq\emptyset}\|e(u)\|^2_{L^2(q''_{0,j}\cup q''_{0,i})}.
\end{split}
\end{equation}

\item[(a2)] If $q_{p,i}\nsubseteq Q_{r_{-1}}$, then for $q_{k,j}\cap q_{p,i}\not=\emptyset$, we decompose
\[w_{p,i}-w_{k,j}=\rho_{\delta_p}*(v_{p,i}-a_{p,i})-\rho_{\delta_k}*(w_{k,j}-a_{k,j})+(a_{p,i}-a_{k,j}).\]
Notice the crucial step that $\rho_{\delta_{k}}*a_{k,j}=a_{k,j}$ due to the fact that $a_{k,j}$ is harmonic (since it is affine).
Then we have  thanks to \eqref{eq:L4} and \eqref{eq:aL4}
\begin{align*}
\Vert \rho_{\delta_k}*(v_{p,i}-a_{p,i})\Vert_{L^2(q_{k,j}\cap q_{p,i})}&\leq\Vert v_{p,i}-a_{p,i}\Vert_{L^2(q_{p,i}')}\lesssim\delta_p\Vert e(u)\Vert_{L^2(q_{i,p}'')}\\
\Vert \rho_{\delta_k}*(v_{k,j}-a_{k,j})\Vert_{L^2(q_{k,j}\cap q_{p,i})}&\leq\Vert v_{k,j}-a_{k,j}\Vert_{L^2(q_{k,j}')}\lesssim\delta_p\Vert e(u)\Vert_{L^2(q_{k,j}'')}\\
\Vert a_{p,i}-a_{k,j}\Vert_{L^2(q_{k,j}\cap q_{p,i})}&\lesssim\delta_p^{ 1+\frac{1}{4}} \Vert e(u)\Vert_{L^2(q_{k,j}''\cup q_{p,i}'')},
\end{align*}
where we also used the fact that $\delta_p$ and $\delta_k$ differ from at most a factor $2$. And so we obtain with the same computations as the previous point that

\begin{equation}
\label{eq:estqpi}
\left\|\sum_{k,j}\nabla \varphi_{k,j} \odot w_{k,j}\right\|^2_{L^2(q_{p,i})}\le  \sum_{q_{k,j}\cap q_{p,i}\neq\emptyset}\|e(u)\|^2_{L^2(q''_{k,j}\cup q''_{p,i})}.
\end{equation}

\end{itemize}

Gathering \eqref{eq:estq0i} and \eqref{eq:estqpi}, and in view of the choice of $r$ which satisfies \eqref{eq:r2}, we deduce 
\begin{equation}
\label{eq:est-2term}
\begin{split}
\left\|\sum_{k,j}\nabla \varphi_{k,j} \odot w_{k,j}\right\|^2_{L^2(Q_r)}&\le \sum_{p,i}\left\|\sum_{k,j}\nabla \varphi_{k,j} \odot w_{k,j}\right\|^2_{L^2(q_{p,i})}\\
&\lesssim \delta^{\frac{1}{2}}\|e(u)\|^2_{L^2(Q_{r_1})}+\|e(u)\|^2_{L^2(Q_r\setminus Q_{r_{-2}})}\lesssim
\delta^{\frac{1}{2}}\|e(u)\|^2_{L^2(Q_1)}.
\end{split}
\end{equation}
 
Coming back to \eqref{eq_ew}, in view of \eqref{eq:est-1term} and \eqref{eq:est-2term} we deduce that
$$
\left\|e(w)-\sum_{k,j}\varphi_{k,j}\rho_{\delta_k}* e(u)\right\|_{L^2(Q_r)}\lesssim \delta^{\frac{1}{6}}\|e(u)\|_{L^2(Q_1)},
$$
so that claim \eqref{eq:claimw2} follows. 
\par
In particular we get also that $w\in H^1(Q_r)$. Claim \eqref{eq:claimw3} concerning the traces follows by the construction which involves 
convolutions whose radius becomes finer and finer as we approach $\partial Q_r$ as detailed in \cite{CCI_19}. Finally we deduce that $w\in SBD(Q_1)$, and that claim \eqref{eq:claimw1} holds true.

\end{proof}

In the proof of Proposition \ref{lem:smoothing1} we made use of the following lemma due to Ne\v{c}as {( see \cite[Theorem IV.3.1]{Boyer}, or also~\cite{Bogovskii})}. 

 \begin{lemma}\label{lem:necas}
Let $\Om$ be a bounded, connected open set  with Lipschitz boundary, and  let $L^2_0(\Om)$ be the set of zero-average $L^2$-functions. Then there is a continuous linear map $\Phi:L^2_0(\Om)\to H^1_0(\Om;\R^d)$ such that $\text{div}\circ \Phi=\text{Id}_{L^2_0(\Om)}$.
\end{lemma}

\subsection{Regularity for quasi  minimizers of the Griffith energy}
\label{sec:griffith}

Let $\Om\subseteq \R^2$ be an open set. In all the following, we will consider the {\it Griffith functional}
\[
G(J,u,B):=\int_{B}|e(u)|^2\,dx+\Ha^1(J\cap B),
\]
where $B\subseteq \Om$ is a Borel set.

We consider the following ({\it very} weak) notion of local minimality.

\begin{definition}[{\bf Quasi  minimizers}]
\label{def:almostquasimin}
 Let $\Lambda,\overline{r}>0$.
We say that $(J,u)\in \mathcal{V}(\Om)$  (recall Definition \ref{def:nubar}) is a $(\Lambda,\overline{r})$ quasi minimizer  of $G$ on ${\mathcal V}(\Om)$  if $G(J,u,\om)<+\infty$ for any open set $\om\Subset\Om$,  and for any square $Q_{x,r}\Subset\Om$ with $r\in (0,\overline{r})$, $\Ha^0(J\cap \partial Q_{x,r})=0$ and 
$$
\limsup_{s\to 0^+}\frac{1}{s}\Ha^1\left(J\cap(Q_{x,r}\setminus Q_{x,r-s})\right)<1,
$$
and for any function $v\in H^1(Q_{x,r};\R^2)$ with ${\rm div}\,v=0$ and $v=u$ on $\partial Q_{x,r}$, we have
\begin{equation}
\label{eq:ineqG}
\int_{Q_{x,r}}|e(u)|^2\,dx+\Ha^1(J\cap Q_{x,r})\leq \int_{Q_{x,r}}|e(v)|^2\,dx+\Lambda r^2.
\end{equation}

\end{definition}

\begin{remark}
 
{\rm

Notice that under the assumption of the previous definition, we have $(J\setminus Q_{x,r},v)\in {\mathcal V}(\Om)$, where we extended $v$ to the entire $\Om$ by setting $v=u$ in $\Om\setminus Q_{x,r}$, and inequality \eqref{eq:ineqG} may be written as 
$$
G(J,u,Q_{x,r})\leq G(J\setminus Q_{x,r},v,Q_{x,r})+\Lambda r^2.
$$ 
\par
The local minimality property involves thus a comparison between $(J,u)$ and very special competitors: the Sobolev function $v$ is obtained by ``smoothing out'' the jumps of $u$ inside suitable squares $Q_{x,r}$, so that it can be paired with the rectifiable set $J\setminus Q_{x,r}$, yielding the admissible pair $(J\setminus Q_{x,r},v)$. Such competitors are provided by the Smoothing Lemma \ref{lem:smoothing1}, for which the dimension two is essential. A somehow related weak notion of minimality involving Sobolev competitors, still in dimension two, has been investigated in \cite{BFG} ({\it minimality with respect to its own jump set}) for the (scalar) Mumford-Shah functional.

}
\end{remark}

\begin{remark}
{\rm
The notion of minimality is weak enough to include any local minimizer of a functional of the form
\[
F(u,A):=\int_{A}|e(u)|^2\,dx+\int_{J_u\cap A}\Theta(\nu_u,u^+,u^-)d\Ha^1
\]
where $\Theta$ is a measurable function such that $\inf(\Theta)\geq 1$ (or, $\inf(\Theta)>0$ up to scaling).
}
\end{remark}

The following result is the key ingredient for obtaining regularity.

\begin{proposition}[\bf Decay estimate]
\label{prop:decay}
 
Let $\Lambda>0$. There exists a universal constant $\overline{\tau}\in (0,1)$ such that for every $\tau \in (0,\bar\tau)$ there exist $\eps=\eps(\tau)$ and $\bar r=\bar r(\tau)$ with the property that for any $(\Lambda,\overline{r})$-quasi minimizer $(J,u)$ of $G$ on ${\mathcal V}(\Om)$, if for  $r<\bar r$
$$
G(J,u,Q_r)\geq r^{3/2}\qquad\text{and}\qquad \Ha^1(J\cap Q_r)\leq \eps r,
$$
then
\[
G(J,u,Q_{\tau r})\leq \tau^{3/2}G({ J},u,Q_r).
\]
\end{proposition}

\begin{proof}
 By contradiction assume that for $\tau$ sufficiently small there exist $\eps_n\to 0$, $\bar r_n\to 0$, $0<r_n<\bar r_n$, and a sequence $(K_n,w_n)$ of 
$(\Lambda,\overline{r}_n)$-minimizers for such that for every $n$
\[
G(K_n,w_n,Q_{r_n})\geq r_n^{3/2},\quad\ \Ha^1(K_{n}\cap Q_{r_n})\leq \eps_n r_n,\quad\text{and}\quad G(K_n,w_n,Q_{\tau r_n})>\tau^{3/2}G(K_n,w_n,Q_{r_n}).
\]
Let
\[
g_n:=G(K_n,w_n,Q_{r_n}),\qquad J_n:=\frac{K_n}{r_n}\qquad\text{and}\qquad u_n(x):=\frac{w_n(r_n x)}{\sqrt{g_n}}.
\]
Then $(J_n,u_n)$ is a $(\Lambda \sqrt{r_n},1)$-minimizer of $G_n(\cdot,\cdot,Q_1)$, where
\[
G_n(J,u,A):=\int_{A}|e(u)|^2\,dx+\frac{r_n}{g_n}\Ha^1(J\cap A),
\]
with
\begin{equation}
\label{eq:abs-gn}
G_n(J_n,u_n,Q_1)=1,\qquad G_n(J_n,u_n,Q_\tau)>\tau^{3/2}\qquad\text{and}\qquad \Ha^1(J_n\cap Q_1)=\eps_n.
\end{equation}

Let us apply the Smoothing Lemma \ref{lem:smoothing1}: if $\delta_n=\eps_n^{\frac{1}{2}}$, let $Q_{s_n}$ with $1-\delta_n^{\frac{1}{2}}<s_n<1$ be the square on which the jumps of $u_n$ are smoothed out giving raise to the function $v_n$, associated to an admissible pair $(J\setminus Q_{s_n},v_n)\in {\mathcal V}(Q_1)$. In particular
\begin{equation}
\label{eq:boundevn}
\|e(v_n)\|_{L^2(Q_1)} \le (1+C\delta_n^{\frac{1}{6}})\|e(u_n)\|_{L^2(Q_1)}
\qquad\text{with}\qquad 
\|e(u_n)\|_{L^2(Q_1)}\le 1,
\end{equation}
and
\begin{equation}
\label{eq:cutoffvn}
\left\|e(v_n)- \varphi_n \rho_{\delta_n}*e(u)\right\|_{L^2(Q_{s_n})}\le C\delta_n^{\frac{1}{6}}\|e(u_n)\|_{L^2(Q_1)},
\end{equation}
where $C>0$ is independent of $n$ and $\varphi_n\in C^\infty(Q_{s_n},[0,1])$ is such that
$\varphi_n=0$ on $Q_{s_n}\setminus Q_{s_n-\delta_n}$, $\varphi_n=1$ on $Q_{s_n-4\delta_n}$.  Since $v_n$ is divergence free and Sobolev on $Q_{s_n}$ we have
\begin{equation}
\label{eq:divvn}
\int_{\partial Q_{s_n}} v_n\cdot \nu\,d\Ha^1=0.
\end{equation}

\par
By the classical Korn inequality on $Q_{s_n}$ there is an antisymmetric affine function $a_n$ such that $\int_{Q_{s_n}}(v_n-a_n)\,dx=0$ and
\[
\int_{Q_{s_n}}|\nabla (v_n-a_n)|^2\,dx\le C_1\int_{Q_{s_n}}|e(v_n)|^2\,dx
\]
for some $C_1>0$ independent of $n$.
We infer that $(v_n-a_n)$ is bounded in $H^1(Q_{s_n})$. Since $s_n\to 1$, we can assume, up to extracting a further subsequence,

\begin{equation}
\label{eq:weakH1vn}
v_n-a_n \weak w\qquad\text{weakly in }H^1_{loc}(Q_1;\R^2)
\end{equation}
for some $w\in H^1(Q_1)$. Since every $v_n-a_n$ has zero divergence, then so does $w$. Moreover $\|e(w)\|_{L^2(Q_1)}\le 1$.

\par
 Let $\psi\in C_c^\infty(Q_1;\R^2)$ have zero divergence, and let $\eta\in C^\infty_c(Q_1,[0,1])$ be a cut-off function such that
$\{\psi\neq 0\}\Subset \{\eta=1\}$.
Let us consider
$$
z_n:=
\begin{cases}
\mathbb{P}_{Q_{s_n}}\Big[(1-\eta)v_n+\eta(a_n+w+\psi)\Big]&\text{in }Q_{s_n}\\
u_n &\text{in }Q_1\setminus Q_{s_n},
\end{cases}
$$
where $\mathbb{P}_{Q_{s_n}}$ denotes the projection on divergence free $H^1(Q_{s_n})$ vector fields which preserves the trace obtained according to Lemma \ref{lem:necas} by considering 
$$
\mathbb{P}_{Q_{s_n}}(u):=u-\Phi_{Q_{s_n}}({\rm div}\,u)
$$
for any $u\in H^1(Q_{s_n};\R^2)$ with a zero mean divergence. Note that $z_n$ is well defined as 
$$
(1-\eta)v_n+\eta(a_n+w+\psi)=v_n\qquad\text{on }\partial Q_{s_n} 
$$
for $n$ large enough, and so its divergence has zero mean thanks to  \eqref{eq:divvn}. 
\par
Since $(J_n\setminus \partial Q_{s_n},z_n)$ is an admissible competitor for $(J_n,u_n)$ according to Definition \ref{def:almostquasimin}, we obtain

\begin{align*}
&G_n(J_n,u_n,Q_{s_n})\\
&\leq \left\Vert e\left(\mathbb{P}_{Q_{s_n}}\Big[(1-\eta)v_n+\eta(a_n+w+\psi)\Big]\right)\right\Vert_{L^2(Q_{s_n})}^2+\Lambda\sqrt{r_n}\\
&\leq \left(\Vert e\left((1-\eta)v_n+\eta(a_n+w+\psi)\right)\Vert_{L^2(Q_{s_n})}+C\Vert\text{div}((1-\eta)v_n+\eta(a_n+w+\psi))\Vert_{L^2(Q_{s_n})}\right)^2
+\Lambda \sqrt{r_n}\\
&\leq \left(\Vert e\left((1-\eta)v_n+\eta(a_n+w+\psi)\right)\Vert_{L^2(Q_{s_n})}+C\Vert\nabla \eta \cdot(w+a_n-v_n)\Vert_{L^2(Q_{s_n})}\right)^2+\Lambda\sqrt{r_n}.
\end{align*}
Since
\[
\Vert\nabla \eta \cdot(w+a_n-v_n)\Vert_{L^2(Q_{s_n})} \to 0
\]
and  (recall that $\{\psi\not=0\}\Subset \{\eta=1\}$)
\begin{align*}
\Vert e\left((1-\eta)v_n+\eta(a_n+w+\psi)\right)\Vert_{L^2(Q_{s_n})}&=\Vert (1-\eta)e(v_n)+\eta e(w+\psi)+\nabla \eta \odot(w+a_n-v_n)\Vert_{L^2(Q_{s_n})}\\
&\leq \Vert (1-\eta)e(v_n)+\eta e(w+\psi)\Vert_{L^2(Q_{s_n})}+o_n
\end{align*}
where $o_n\to 0$, we infer  thanks to \eqref{eq:boundevn} (and since $e(a_n)=0$)
\begin{equation}
\label{eq:unzn}
G_n(J_n,u_n,Q_{s_n})\leq \Vert (1-\eta)e(v_n-a_n)+\eta e(w+\psi)\Vert_{L^2(Q_{s_n})}^2+o_n.
\end{equation}
Now,  still using \eqref{eq:boundevn}  we may write

\[
\int_{Q_{s_n}}|e(v_n-a_n)|^2\,dx\leq (1+C\delta_n^{\frac{1}{6}})^2\int_{Q_{s_n}}|e(u_n)|^2\,dx\leq (1+C\delta_n^{\frac{1}{6}})^2G_n(J_n,u_n,Q_{s_n})  + o_n,
\]

and so  coming back to \eqref{eq:unzn} we deduce
\begin{equation*}
\label{eq:znznz}
\|e(v_n-a_n)\|^2_{L^2(Q_{s_n})}\leq \Vert (1-\eta)e(v_n-a_n)+\eta e(w+\psi)\Vert_{L^2(Q_{s_n})}^2+o_n.
\end{equation*}
 
This yields
\begin{align*}
\int_{Q_{s_n}}\left(1-(1-\eta)^2\right)|e(v_n-a_n)|^2\,dx\leq \int_{Q_{s_n}}\left(2\eta(1-\eta)e(v_n-a_n):e(w)+\eta^2|e(w+\psi)|^2\right)\,dx+o_n
\end{align*}
so that  in view of \eqref{eq:weakH1vn}
\begin{align*}
\int_{Q_{1}}\left(1-(1-\eta)^2\right)|e(w)|^2\,dx&\leq \limsup_{n\to\infty}\int_{Q_{1}}\left(1-(1-\eta)^2\right)|e(v_n-a_n)|^2\,dx\\
&\le \int_{Q_{1}}\left(2\eta(1-\eta)e(w):e(w)+\eta^2|e(w+\psi)|^2\right)\,dx.
\end{align*}
Notice that by choosing $\psi=0$ and letting $\eta$ localize on characteristic functions of open sets, we infer that

\begin{equation}
\label{eq:strongew}
e(v_n-a_n)\to e(w)\qquad \text{strongly in }L^2_{loc}(Q_1;M^{2\times 2}_{sym}).
\end{equation}

In particular we get
\[
\int_{Q_1}|e(w)|^2\,dx\leq\int_{Q_1}|e(w+\psi)|^2\,dx,
\]

which means that $w$ is a local minimizer of the energy $z\mapsto \|e(z)\|_{L^2(Q_1)}^2$ on $H^1$ functions with zero divergence. This yields
$\Delta w=\nabla p$ for some $p\in L^2(Q_1)$.  Using the Lemma \ref{lem:regbil} below, we have  
\[
\int_{Q_{\tau }}|e(w)|^2\,dx\leq \frac{1}{2}\tau^{\frac{3}{2}}\int_{Q_1}|e(w)|^2\,dx\leq \frac{1}{2}\tau^{\frac{3}{2}}.
\]

Taking into account \eqref{eq:strongew} we deduce
\begin{equation}
\label{eq:vnqtau}
\|e(v_n)\|^2_{L^2(Q_{\tau+\delta_n})}\le \frac{1}{2}\tau^{\frac{3}{2}}+o_n.
\end{equation}
By minimality we have
$$
G(u_n, J_n,Q_{s_n})\le \|e(v_n)\|^2_{L^2(Q_{s_n})}+\Lambda\sqrt r_n=\|e(v_n)\|^2_{L^2(Q_{\tau +\delta_n})}+
\|e(v_n)\|^2_{L^2(Q_{s_n}\setminus Q_{\tau +\delta_n})}+\Lambda\sqrt r_n
$$
while thanks to \eqref{eq:cutoffvn}
\begin{multline*}
\|e(v_n)\|_{L^2(Q_{s_n}\setminus Q_{\tau +\delta_n})}\le \|e(v_n)-\varphi_n \rho_{\delta_n}*e(u_n)\|_{L^2(Q_{s_n}\setminus Q_{\tau +\delta_n)}}+
\|\varphi_n \rho_{\delta_n}*e(u_n)\|_{L^2(Q_{s_n}\setminus Q_{\tau +\delta_n})}\\
\le o_n+\|\rho_{\delta_n}*e(u_n)\|_{L^2(Q_{s_n-\delta_n}\setminus Q_{\tau +\delta_n})}
\le o_n+\|e(u_n)\|_{L^2(Q_{s_n}\setminus Q_{\tau })}.
\end{multline*}
In view of \eqref{eq:vnqtau} we infer
\begin{multline*}
G(u_n,J_n,Q_{s_n})\le \|e(v_n)\|^2_{L^2(Q_{\tau +\delta_n})}+[o_n+\|e(u_n)\|_{L^2(Q_{s_n}\setminus Q_{\tau })}]^2
+\Lambda\sqrt{r_n}
\le \frac{1}{2}\tau^{\frac{3}{2}}+\tilde o_n+\|e(u_n)\|^2_{L^2(Q_{s_n}\setminus Q_{\tau })}
\end{multline*}
so that
\[
G(u_n,J_n,Q_\tau)\le  \frac{1}{2}\tau^{\frac{3}{2}}+\tilde o_n.
\]
In conclusion, taking into account \eqref{eq:abs-gn}, if $n$ is large enough we get
\[
\tau^{\frac{3}{2}}< G(u_n,J_n,Q_\tau)\le  \frac{1}{2}\tau^{\frac{3}{2}}+\tilde o_n,
\]
which is a contradiction.

\end{proof}

In the preceding proof, we made use of the following result.

\begin{lemma}
\label{lem:regbil}
There exists a constant $C_0>0$ such that for any divergence-free vector field $u\in H^1(Q_1;\R^2)$ such that $\Delta u=\nabla p$ for some pressure $p\in L^2(Q_1)$, we have
\[
\forall \tau\in (0,1/2]\,:\, \ \int_{Q_{\tau}}|e(u)|^2\,dx\leq C_0\tau^2\int_{Q_1}|e(u)|^2\,dx.
\]
In particular, for any $0<\tau\leq \overline{\tau}:=\frac{1}{4C_0^2}\wedge\frac{1}{2}$ we have 
$$
\int_{Q_{\tau}}|e(u)|^2\,dx\leq \frac{1}{2}\tau^{3/2}\int_{Q_1}|e(u)|^2\,dx.
$$
\end{lemma}

\begin{proof}
Notice that $e(u)$ is invariant by the addition of an asymmetric affine function $a$. Up to a translation by such a function, Korn's inequality tells us that
\[\int_{Q_1}u^2\,dx\leq C\int_{Q_1}|e(u)|^2\,dx.\]
The equations verified by $u$ are equivalent to the existence of $\varphi\in H^2(Q_1)$ such that $\varphi(0)=0$, $u=\nabla^\bot\varphi$, and $\Delta^2\varphi=0$. By elliptic regularity there is a constant $C'$ such that
\[\sup_{Q_{1/2}}\left|\nabla^2 \varphi\right|^2\leq C'\int_{Q_1}|\nabla\varphi|^2\,dx\]
and so for any $\tau\leq 1/2$,
\begin{align*}
\int_{Q_\tau}|e(u)|^2\,dx&\leq 4|Q_\tau|\sup_{Q_{1/2}}\left|\nabla^2 \varphi\right|^2\leq 4CC'|Q_1|\tau^2\int_{Q_1}|e(u)|^2\,dx.
\end{align*}
\end{proof}

 The decay estimate can be iterated as follows.

 \begin{lemma}[\bf Iteration of the decay]
 \label{lem:iteration}
  Let $\Lambda>0$ and, according to Proposition \ref{prop:decay},  let $\tau_0$ be small enough such that the decay estimate applies
 { with $\eps_0=\eps(\tau_0)$ and $\bar r_0=\bar r(\tau_0)$, and}
  let $\tau_1\in (0,\eps_0^2)$ be small enough that the decay property applies { with 
 $\eps_1$, $\bar r_1$.}
 Finally, let
\[
\bar r:=\min\left(\bar r_0,\bar r_1,\eps_0^2\tau_1^2,\frac{\eps_0^2\tau_0^3}{ \tau_1 }\right).
\]
Suppose that $(J,u)$ is a $(\Lambda, \bar r )$-quasi minimizer  of $G$ on ${\mathcal V}(\Om)$  and $G(J,u,Q_{x,r})\leq \eps_1 r$ for some $r\in (0, \bar r)$. Then for all $k\in\mathbb{N}$,
\[
G(J,u,Q_{x,\tau_0^k \tau_1 r})\le \eps_ 0\tau_0^{\frac{3}{2}k}\tau_1 r.
\]
\end{lemma}
\begin{proof}  Let us prove the statement by induction on $k$.  In the following, we write $g(r)=G(J,u,Q_{x,r})$,  so that we need to check that if $g(r)\le \eps_1 r$, then for every $k\in\N$
\begin{equation}
\label{eq:gk}
g(\tau_0^k \tau_1 r)\le \eps_ 0\tau_0^{\frac{3}{2}k}\tau_1 r.
\end{equation}

 The inequality is true for $k=0$. Indeed we have the following alternatives:
\begin{itemize}
\item[(a)] If $g(r)>r^{3/2}$ then $g(\tau_1 r)\leq  \tau_1^{3/2}g(r) \le \sqrt{\tau_1}\tau_1\eps_1 r \leq \eps_0 \tau_1 r$ by definition of $\tau_1$.
\item[(b)] If $g(r)\leq r^{3/2}$, then $g(\tau_1 r)\leq g(r)\le r^{3/2}\leq\eps_0 \tau_1 r$ by definition of $\bar r$.
\end{itemize}
 Assume now that \eqref{eq:gk} holds.  Notice that by definition of $G$ we have  (since $\tau_0<1$)
\[
\Ha^1(J\cap Q_{\tau_0^k \tau_1 r})\leq g(\tau_0^k\tau_1 r)\le \eps_ 0\tau_0^{\frac{3}{2}k}\tau_1 r\leq \eps_0 \tau_0^k\tau_1 r,
\]
so the decay property of Proposition \ref{prop:decay} may be applied. Again we  have two alternatives.
\begin{itemize}
\item[(a)] If $g(\tau_0^k\tau_1 r)>(\tau_0^{ k}\tau_1r)^{3/2}$, by the decay property we have,  using \eqref{eq:gk},
\[
g(\tau_0^{k+1}\tau_1r)\leq \tau_0^{3/2}g(\tau_0^k\tau_1r)\leq \eps_0\tau_{0}^{\frac{3}{2}(k+1)}\tau_1 r.
\]
\item[(b)] If $g(\tau_0^k\tau_1r)\leq (\tau_0^k\tau_1r)^{3/2}$ then  by the definition of $\bar r$
\[g(\tau_0^{k+1}\tau_1r)\leq g(\tau_0^{k}\tau_1r)\leq \sqrt{\frac{\tau_1 r}{\eps_0^2\tau_0^3}}\eps_0\tau_0^{\frac{3}{2}(k+1)}\tau_1 r\leq \eps_0\tau_0^{\frac{3}{2}(k+1)}\tau_1 r.\]
\end{itemize}
 In both cases, \eqref{eq:gk} follows for the choice $k+1$, so that the induction step is proved.
\end{proof}

If we want to draw some conclusions on the regularity of quasi minimizers $(J,u)$, we need somehow to bound the freedom connected to the choice of $J$: notice indeed that any pair $(J\Delta N,u)$ with $\Ha^1(N)=0$ is essentially equivalent to $(J,u)$, where $A\Delta B$ denotes the symmetric difference of sets.
\par
We set
\begin{equation}\label{eq_density}
J^+:=\left\{x\in \Om:\limsup_{r\to 0}\frac{\Ha^1(J\cap Q_{x,r})}{r}>0\right\}.
\end{equation}
$J^+$ is a sort of normalized version of $J$, where points of density zero have been erased.
\par
By standard properties of rectifiable sets we have 
$$
\Ha^1(J\Delta J^+)=0.
$$ 
As a consequence if $(J,u)\in \mathcal V(\Om)$, then also $(J^+,u)\in \mathcal V(\Om)$ with $G(J,u,A)=G(J^+,u,A)$ for every Borel set $A\subseteq \Om$.

\begin{proposition}
\label{prop:lower-bound}

Given $\Lambda>0$, there exist $\eps,\bar r>0$ such that of any $(\Lambda,\overline{r})$-quasi minimizer $(J,u)$  of $G$ on ${\mathcal V}(\Om)$, if 
$G(J,u,Q_{x,r})\leq \eps r$ for some $Q_{x,r}\Subset \Om$ with $r<\bar r$, then
$J^+\cap Q_{x,\frac{r}{2}}=\emptyset$.

\end{proposition}

\begin{proof}
 Let $\eps_0,\eps_1,\tau_1,\tau_2,\bar r$ be given according to Lemma \ref{lem:iteration}. Notice that if $G(J,u,Q_{x,r})\leq \eps_1 r$ with $r<\bar r$, then for any $\rho\in (0,r)$
\begin{equation}
\label{eq:decay-it}
G(J,u,Q_{x,\rho})\leq C_0r^{-\frac{1}{2}}\rho^\frac{3}{2}\qquad \text{where }C_0:=\max\left\{\eps_1\tau_1^{-\frac{3}{2}},\eps_0\tau_0^{ -\frac{1}{2}}\tau_1^{-\frac{1}{2}}\right\}.
\end{equation}
Let us set $\eps:=\frac{1}{2}\eps_1$, and assume  $G(J,u,Q_{x,r})\leq \eps r$. Notice that for any $y\in  Q_{x,\frac{r}{2}}$, we have 
$$
G(J,u,Q_{y,\frac{r}{2}})\le G(J,u,Q_{x,r})\leq\eps r= \eps_1 \frac{r}{2}
$$ 
so that from \eqref{eq:decay-it}

$$
0=\lim_{\rho\to 0^+}\frac{G(J,u,Q_{y,\rho})}{\rho}\ge \limsup_{\rho\to 0^+}
\frac{\Ha^1(J\cap Q_{y,\rho})}{\rho},
$$ 
which yields $J^+\cap Q_{x,\frac{r}{2}}=\emptyset$.

 \end{proof}

 \begin{proposition}[\bf Regularity for quasi minimizers]
 \label{prop:Jreg}
Let $\Lambda,\overline{r}>0$.  Then for any $(\Lambda,\overline{r})$-quasi minimizer $(J,u)$ of $G$  on ${\mathcal V}(\Om)$  we have that $J^+$ (see \eqref{eq_density}) is essentially closed in $\Om$, i.e.,
\[
\Ha^1\left(\Om\cap (\overline{J^+}\setminus J^+) \right)=0,
\]
while $u\in  C^\infty(\Om\setminus \overline{J^+})$.

\end{proposition}

\begin{proof}
Since the functional $G$ coincides with a volume integral outside $J$, there exists a $\Ha^1$-negligible set $N\subset\Om\setminus J$ such that for every $x\in \Om\setminus (J\cup N)$ we have
\[
\lim_{\rho\to 0}\frac{G(J,u,Q_{x,\rho})}{\rho}=0.
\]
 Thanks to Proposition \ref{prop:lower-bound} we infer
$$
\Om\cap\overline{J^+}\subset J\cup N\subset J^+\cup(J\setminus J^+)\cup N.
$$
Since the last two sets are $\Ha^1$-negligible, we infer $\Ha^1(\Om\cap(\overline{J^+}\setminus J^+))=0$. 
\par
 Since 
$$
\Ha^1(J_u\setminus \overline{J^+})\le \Ha^1(J\setminus \overline{J^+})=0,
$$ 
we get that $u$ is locally $H^1$ on $\Om\setminus \overline{J^+}$ (thanks to Korn's inequality): smoothness then follows from the regularity theory for solutions to Stokes equation (see e.g. \cite[Theorem IV.5.8]{Boyer}).

\end{proof}

\subsection{Proof of Theorem \ref{main2}}
\label{sec:proof2}

We are now in a position to prove the regularity result given by Theorem \ref{main2}.

\par
 Let $(E,u)$ be a minimizer of $\J$ and let us set 
$$
\Lambda:=4\text{Lip}(f) \qquad\text{and}\qquad J:=J_u\cup\partial^* E.
$$ 
We also assume (up to multiplying $u$ by $c^{-\frac{1}{2}}$) that the constant $c$ of \eqref{eq_JEu} is $1$. 
\par
We first prove that $(J,u)$ is a  $(\Lambda,1)$ quasi minimizer of the Griffith functional $G$ on ${\mathcal V}(\Om)$ according to  Definition \ref{def:almostquasimin}.  Indeed, let $Q_{x,r}\Subset \Om$ with $r<1$ be a square as in Definition \ref{def:almostquasimin}, with associated competitor $(J\setminus Q_{x,r},v)$. We claim that either
\begin{equation}
\label{eq:disconnected}
\Ha^1(\partial Q_{x,r} \setminus E^{(1)})=0\qquad\text{or}\qquad \Ha^1(\partial Q_{x,r} \setminus E^{(0)})=0.
\end{equation}
In the first case,  from the minimality inequality
$$
\J(E,u)\le \J(E,u1_{\Om \setminus Q_{x,r}})
$$
we deduce $u=0$ a.e. on $Q_{x,r}$ and $\Ha^1(\partial^*E\cap Q_{x,r})=0$, so that the inequality to check for quasi minimality is trivially satisfied.
Notice that admissibility of $(E,u1_{\Om \setminus Q_{x,r}})$ for the main problem follows from the fact that the trace of $u$ on $\partial Q_{x,r}$ is zero, being that boundary composed of points of density one of the set $E$ on which $u$ vanishes.

\par
 If the second possibility in \eqref{eq:disconnected} holds true, then the relations (see \cite[Theorem 16.3]{Maggi_12} and recall that $\Ha^0(\partial Q_{x,r} \cap \partial^*E)=0$ by the properties of $Q_{x,r}$)
$$
\J(E,u)\le \J(E\setminus Q_{x,r},v)
\qquad\text{and}\qquad
\partial^*(E\setminus Q_{x,r})=\partial^*E \setminus \overline{Q_{x,r}}=\partial^*E \setminus Q_{x,r}
$$
yield in particular
\[
\int_{Q_{x,r}}|e(u)|^2\,dx+\Ha^1(J\cap Q_{x,r})\leq \int_{Q_{x,r}}|e(v)|^2\,dx+\Lambda r^2,
\]
so that the quasi minimality of $(J,u)$ follows.
\par
By Proposition \ref{prop:Jreg}, we get that the normalized set $J^+$ (see \eqref{eq_density}) is essentially closed in $\Om$, i.e.,
$$
\Ha^1(\Om\cap(\overline{J^+}\setminus J^+))=0,
$$ 
 and $u$ is smooth on $\Om\setminus \overline{J^+}$, so that 
$$
\Om \cap J_u\subseteq  \overline{J^+}.
$$ 
On the other hand, in view of the general properties of the reduced boundary of sets of finite perimeter (see \cite[Theorem 3.59]{AFP} or \cite[Theorem 15.5]{Maggi_12}) we have $\partial^*E \subseteq J^+$.  Taking into account that $\Ha^1(J^+\Delta J)=0$ (where $\Delta$ denotes the symmetric difference of sets)   we infer
\[\Ha^{1}(\Om\cap\overline{J_u \cup  \partial^*E}\setminus (J_u \cup \partial^*E))\leq \Ha^{1}(\Om\cap \overline{J^+}\setminus J)\leq \Ha^{1}(\Om\cap\overline{J^+}\setminus J^+)+\Ha^{1}(J^+\setminus J)=0\]
so that the conclusion follows.
\par
In order to complete the proof, we need to check claim \eqref{eq:disconnected}.
\par\noindent
Assume by contradiction that the claim is false. Then there exists $p\in E^{(1)}\cap\partial Q_{x,r}$ and $q\in E^{(0)}\cap\partial Q_{x,r}$ that are not in one of the corners. Without loss of generality we suppose $p,q\in \{x-r e_2+\R e_1\}$ with $p_1<q_1$, the case when both are in different sides being analog. We let  for $s>0$ small
\begin{itemize}
\item $C_p:=p+[-s,0]\times [0,s]$ and $C_q:=q+[0,s]^2$, 
\item $g_s:[p_1-s,q_1+s]\to [0,1]$  be zero at the extremes, affine on $[p_1-s,p_1]$ and $[q_1,q_1+s]$ and equal to $1$ on $[p_1,q_1]$,
\item  $f_s \in C^1_c(]0,s[)$ with $0\le f_s\le 1$, 
\item $\varphi_s(x)=g_s(x_1)f_s(x_2+ r)$.
\end{itemize}
Then
\begin{align*}
\Ha^1(J\cap (Q_{x,r}\setminus Q_{x,r-s}))&\geq \int_{\partial^*E}\varphi_s (\nu_E)_1\, d\Ha^1=\int_{E}\partial_{1}\varphi_s\, d\Ha^1\\
&=\frac{1}{ s}\int_{E\cap C_p}f_s(y_2+ r)dy-\frac{1}{ s }\int_{E\cap C_q}f_s(y_2+ r)dy
\end{align*}

so that, letting $f_s \nearrow 1$ we get
$$
\frac{\Ha^1(J\cap (Q_{x,r}\setminus Q_{x,r-s}))}{s}\geq \frac{|E\cap C_p|}{|C_p|}-\frac{|E\cap C_q|}{|C_q|}.
$$

Since as $s\to 0^+$,  by assumption on the density properties of $p$ and $q$,  we have
$$
\frac{|E\cap C_p|}{|C_p|}\to 1\qquad\text{and}\qquad \frac{|E\cap C_q|}{|C_q|}\to 0,
$$
we infer
$$
\limsup_{s\to 0}\frac{\Ha^1(J\cap (Q_{x,r}\setminus Q_{x,r-s}))}{s}\geq 1
$$ 
which is against the assumption on $r$ in Definition \ref{def:almostquasimin} of quasi minimality. The proof is thus concluded.

\subsection{Some remarks on a ``strong'' formulation of the problem}
\label{sec:strong}

In this section we elaborate on a different relaxation of the drag minimization problem which involves topologically closed (but not necessarily regular) obstacles $F$ in the channel $\Om$ and velocity vector fields which are $H^1_{loc}$ on $\Om\setminus F$.
\par
Within this perspective, given $\Om\subset \R^d$ open and bounded, it is natural to start with pairs $(F,u)$ such that 

\begin{equation}
\label{eq:topology}
\text{$F\subseteq \Om$ is relatively closed},\qquad \text{ $\Om\cap\partial F$ is rectifiable},\qquad \Ha^{d-1}(\Om\cap\partial F)<+\infty
\end{equation}
and
\begin{equation}
\label{eq:H1loc}
u\in H^1_\loc(\Om\setminus F;\R^d), \qquad \text{$\text{div}\,u=0$ in $\Om\setminus F$,}\qquad  e(u)\in L^2(\Om\setminus F;M^{d\times d}_{sym}).
\end{equation}
Notice that, as for the relaxation studied in the previous sections,  $\partial F$ may contain ``lower dimensional'' parts.
The set $\Om\setminus F$ is open, so that the space $H^1_{loc}(\Om\setminus F;\R^d)$ is well defined. 
\par
It is not clear how to talk about traces on $\partial(\Om\setminus F)$, which are fundamental to formulate the tangency constraint, as the set is in general not regular. It turns out that velocities admit a well defined trace on $\Ha^{d-1}$ almost every point $\partial F$ even if this set is not assumed to be only rectifiable and not regular. This is 
a consequence of the following result which involves the space $GSBD$ of {\it Generalised Functions of Bounded Deformations} introduced in \cite{M13}.
Let us set
\begin{equation}
\label{eq:utilde}
\tilde{u}:=
\begin{cases}
u& \text{ in }\Om\setminus F\\ 
0& \text{ in } F.
\end{cases}
\end{equation}

\begin{lemma}
 
Let $\Om\subseteq \R^d$ be a bounded open set, and assume that the pair $(F,u)$ satisfies \eqref{eq:topology} and \eqref{eq:H1loc}.
Then $\tilde{u}\in GSBD(\Om)$ with $\Ha^{d-1}(J_{\tilde{u}}\setminus \partial F)=0$. 

\end{lemma}

\begin{proof}
Since $\Ha^{d-1}(\Om\cap\partial F)<\infty$,  for every $\eps>0$ we may find some covering of $\partial F$ through a finite union of balls of radius less than $\eps$, denoted $(B^\eps_i)_{1\leq i\leq N^\eps}$, such that 
$$
\sum_{i=1}^{N^\eps}\left(\frac{\text{diam}(B_i^\eps)}{2}\right)^{d-1}\leq C
$$ 
for some $C>0$ that does not depend on $\eps$. Let $B^\eps$ be the union of these balls - which is a Lipschitz set up to a small perturbation of the radii - and let $u^\eps:=u1_{\Om\setminus B^\eps}$. Then $u^\eps\in SBD(\Om)$ with
\[
Eu^\eps=e(u)\,dx\lfloor (\Om\setminus (F\cup B^\eps))+u\Ha^{d-1}\lfloor \partial B^\eps.
\]
Moreover
\[
u^\eps\to \tilde{u}\qquad\text{a.e. in $\Om$}
\]
with
\[
\limsup_{\eps\to 0}\int_{\Om}|e(u^\eps)|^2\,dx+\Ha^{d-1}(J_{u^\eps})<+\infty.
\]
We apply \cite[Theorem 1.1]{CC20} to $(u^\eps)$: since $\tilde{u}$ is finite almost everywhere, we directly identify $\tilde{u}$ with the limit that is obtained, and we infer $\tilde{u}\in GSBD(\Om)$; moreover up to a $\Ha^{d-1}$-negligible set, $J_{\tilde{u}}\subset \partial F$ by construction, and the result follows.
\end{proof}

Coming back to configurations $(F,u)$ satisfying \eqref{eq:topology} and \eqref{eq:H1loc}, up to a choice of orientation of the rectifiable set $\Om\cap\partial F$,
there is no ambiguity in defining the traces $u^\pm_{|\partial F}$ of $u$ on $\Ha^{d-1}$-almost all points of $\Om\cap \partial F$. 
\par
 In addition to the previous items, we thus require also for $(F,u)$ the non-penetration condition
\begin{equation}
\label{eq:tangencyF}
u^\pm_{|\partial F}\cdot \nu_{\partial F}=0\qquad\text{$\Ha^{d-1}$-a.e. on $\Om\cap \partial F$.}
\end{equation}

Given an admissible configuration $(F,u)$, we can consider the following energy  (all the constants have been normalized to $1$)
\begin{align*}
J(F,u):=&\int_{\Om\setminus F} |e(u)|^2\,dx+\int_{\Om\cap\partial^eF}|u^+|^2\,d\Ha^{d-1}+\int_{\Om\cap F^{(0)}\cap F}[|u^+|^2+|u^-|^2]\,d\Ha^{d-1}\\
&+\Ha^{d-1}(\Om\cap\partial^eF)+2\Ha^{d-1}(\Om\cap  F^{(0)}\cap F)+f(|F|),
\end{align*}
 where $\partial^eF$ denotes the {\it measure theoretical boundary} of $F$, and $f$ is the penalization function introduced in the previous sections (see \eqref{eq:f}).

Configuration with finite energy are linked to admissible configurations of our main relaxed problem by the following result.

\begin{lemma}
\label{lem:str-sbd}
Let $\Om\subseteq \R^d$ be open and bounded, and let $(F,u)$ satisfy \eqref{eq:topology} and \eqref{eq:H1loc} with $J(F,u)<+\infty$. Then the function $\tilde u$ defined in \eqref{eq:utilde} is such that $\tilde u\in SBD(\Om)$.
\end{lemma}

\begin{proof}
It suffices to note that for every direction $\xi\in\mathbb{S}^{d-1}$ we have 
\[
J(F,u)\geq \int_{\xi^\bot}\left[\int_{\Om_y^\xi}|(\tilde{u}_y^\xi)'(t)|^2dt+\sum_{t\in J_{\tilde{u}_y^\xi}}\left(1+|(\tilde{u}_y^\xi)^+(t)|^2+|(\tilde{u}_y^\xi)^-(t)|^2\right)\right]d\Ha^1(y).
\]
\end{proof}

Dealing with boundary conditions yields to the same problem highlighted in our main relaxation. Assume $\Om$ has a Lipschitz boundary, and let us write simply $u$ in place of $\tilde u$. We have that $u\in SBD(\Om)$ so that the trace on $\partial\Om$ is well defined. Given a divergence free vector field $V\in C^1(\R^d;\R^d)$, we can deal with the relaxation of the boundary condition by considering the set
$$
\Gamma_{u,V}:=\{x\in\partial \Om\,:\, u(x)\not=V(x)\},
$$
and enforcing the non-penetration constraint leading to
\begin{equation}
\label{eq:Gamma}
u\cdot \nu_{\partial\Om}=0\quad\text{and}\quad V\cdot \nu_{\partial\Om}=0\qquad\text{$\Ha^{d-1}$-a.e. on $\Gamma_{V,\partial\Om}$}.
\end{equation}
So for a configuration $(F,u)$ satisfying \eqref{eq:topology}, \eqref{eq:H1loc}, \eqref{eq:tangencyF} and \eqref{eq:Gamma}, we can consider the energy
$$
\J^{\text{strong}}(F,u):=J(F,u)+2\Ha^{d-1}(\Gamma_{u,V})+\int_{\Gamma_{u,V}}[|V|^2+|u|^2]\,d\Ha^{d-1}.
$$
The minimization of $\J^{\text{strong}}$ on admissible configurations is a different possible relaxation of the original drag minimization problem. We clearly have
$$
\min_{(E,u)\in \mathcal{A}_V(\Om)}\J(E,u)\le \inf_{(F,u)}\J^{\text{strong}}(F,u).
$$
Equality is reached in dimension two thanks to the regularity result given by Theorem \ref{main2}. Indeed, if $(E,u)$ is a minimizer for $\J$, we know that $\partial^*E\cup J_u$ is essentially closed, so that an admissible relatively closed set $F$ arises by considering the complement of the union of the connected components of 
$\Om\setminus \overline{\partial^*E\cup J_u}$ on which $u$ does not vanish identically. The function $u$ is smooth outside $F$, so that the pair $(F,u)$ is strongly admissible with $\J^{\text{strong}}(F,u)=\J(E,u)$. As a consequence, in dimension two the relaxed problem
$$
 \min_{(F,u)}\J^{\text{strong}}(F,u)
 $$
 is indeed well posed.

\bibliographystyle{plain}
\bibliography{ag2021}

\end{document}